\documentclass[11pt]{article}
\usepackage{epsfig}
\usepackage{amssymb,amsmath,amsthm,amscd}
\usepackage{latexsym}
 \usepackage{bbm,dsfont}

\newtheorem{thm}{Theorem}[section]
\newtheorem{prop}[thm]{Proposition}

\newtheorem{lem}[thm]{Lemma}

\theoremstyle{definition}

\newtheorem{rem}[thm]{Remark}

\renewcommand{\theequation}{\arabic{section}.\arabic{equation}}

\newcommand{\be}{\begin{equation}}
\newcommand{\ee}{\end{equation}}

\newcommand{\R}{\mathbb{R}}
\newcommand{\N}{\mathbb{N}}
\newcommand{\E}{\mathbb{E}}
\newcommand{\1}{\mathds{1}}

\def \eps {{ \varepsilon }}

\def \o {{  {\mathbb{R}^n} }}

\def \calf {{  {\mathcal{F}} }}

\def \cala {{  {\mathcal{A}}  }}

\def \calg {{  {\mathcal{G}}  }}

\def \caly {{  {\mathcal{Y}}  }}

 \def \call {{  {\mathcal{L}}  }}

\begin{document}

\baselineskip=1.2 \baselineskip

\title{
       {\large\bf {  Large Deviations  of Invariant Measures
  of Fractional 
 Reaction-Diffusion  Equations
 Driven by  Multiplicative  Noise
 on $\R^n$
        }}
      }
 \author{Bixiang Wang$^{1}$\thanks{E-mail address: bwang@nmt.edu}, \,\,
  Lu Yang$^{2}$ \thanks{E-mail address: yanglu@lzu.edu.cn}
\\
\normalsize{$^1$Department of Mathematics, New Mexico Institute of Mining and Technology} \\
\normalsize{Socorro, NM 87801, USA} \\
\normalsize{$^2$School of Mathematics and Statistics, Lanzhou University, 
Lanzhou 730000, China} \\
\\ } 
\date{}
 \maketitle
 
   {\bf Abstract.}
 In this paper, we establish 
large deviations of  invariant measures
for fractional stochastic reaction-diffusion equations 
with polynomial reaction term driven by 
unbounded multiplicative
noise on  $\R^n$. 
Under a  dissipativity condition,
 we first prove  the  Freidlin-Wentzell large deviation
 principle 
 of solution trajectories uniformly  over all initial data in
$L^2(\mathbb R^n)$. The key ingredient is a positive-time 
$L^\infty$-estimate
for differences of controlled solutions which  is independent of 
all initial
data and hence 
yields uniform convergence of controlled stochastic solutions over the entire
phase space. 
Then using
the
uniform trajectory estimates, contraction and a
second-moment estimate,
we further prove the large deviations
of   invariant measures of the 
equation  with a good quasipotential
without introducing weighted spaces or 
establishing exponential tightness.
The non-compactness of Sobolev embeddings
in $\R^n$ 
is circumvented by uniform
spatial tail-ends  estimates.
 This is  the 
first result
 on large deviations 
of invariant measures
 of the equation with multiplicative
noise on unbounded domains.

 \vspace{2mm}
 

 \maketitle 
%
%
%
%
%
%
%
%

\section{Introduction}
\setcounter{equation}{0}
 
       Large deviation principles describe the exponential asymptotics of rare events
in stochastic dynamical systems subject to small random perturbations. For
stochastic partial differential equations, large deviations of solution
trajectories quantify departures from deterministic dynamics over finite time
intervals, whereas large deviations of invariant measures characterize rare
fluctuations in the stationary regime. Connecting these two descriptions requires
control of the solution with respect to its initial state and sufficient
compactness to obtain convergence of solutions from weak
topology to strong topology. These requirements become particularly
demanding on unbounded spatial domains, where standard compactness arguments
available on bounded domains fail. The purpose of this paper is to establish
both types of large deviations for fractional stochastic reaction-diffusion
equations on \(\mathbb R^n\), with particular emphasis on trajectory estimates
that hold uniformly over the entire unweighted phase space.

   Consider the  nonlocal 
   stochastic equation on $\R^n$
   for $t>0$:
\begin{equation}\label{intr1}
\begin{aligned}
du^\varepsilon(t)
+\bigl ( (-\Delta)^\alpha u^\varepsilon(t)
  +\lambda u^\varepsilon(t)
  +F(t,x,u^\varepsilon(t))\bigr )\,dt
  =g(t,x)\,dt
  +\sqrt{\varepsilon}\,\sigma(t,x,u^\varepsilon(t))\,dW(t),
\end{aligned}
\end{equation}
with initial condition
\begin{equation}\label{intr2}
u^\varepsilon(0,x)=u_0(x),
\qquad x\in\mathbb R^n,
\end{equation}
where  
$\lambda \ge 0$ is a damping parameter,
$\eps \in (0,1)$ is the noise intensity,
   $g\in L^2_{loc}(\R; L^2(\o))$
   is an external forcing, 
    $(-\Delta)^\alpha$ with $\alpha
 \in (0,1)$ is  the fractional Laplacian,
 and 
    $W$
 is a
 cylindrical Wiener process  in $\ell^2$
 on
 a  complete filtered probability space 
$(\Omega, \mathcal{F}, 
   \{ { \mathcal{F}} _t\} _{t\ge 0},  P )$.   
 The reaction term $F$  satisfies polynomial coercivity
and a quantitative monotonicity condition
with exponent
$p>4$, 
and the  noise
coefficient
$\sigma$
 has the spatially localized form
$$
\sigma(t,x,u)=\sigma_1(t,x)+\kappa(x)\sigma_2(t,x,u),
\qquad
\kappa\in L^2(\mathbb R^n)\cap L^\infty(\mathbb R^n),
$$
under the assumptions stated in Section 3. In particular, the corresponding
Hilbert--Schmidt operator is globally Lipschitz in $u$, but need not be bounded.
For the large deviations of invariant measures, we 
will consider the autonomous
version of \eqref{intr1} 
 and impose an additional  dissipativity condition.

      Fractional  PDEs
 arise in physics to model
 anomalous diffusion,  nonlocal  interactions
 and memory effects
    \cite{abe1,  garr1,    guan2,  
    jara1, kos1}. 
    Such equations have been studied 
    in   
   \cite{abe1,  caff1,  dine1,  gal1, garr1,  guan2,
    jara1, kos1,  luh1,   ros1, ser1, ser2}  
    for the existence and uniqueness of solutions,
    and in
     \cite{chen1, gu1, luh2,  wan10, wangJDE2019}
     for asymptotic dynamics.
     The goal of this paper is to investigate
     the large deviation principle (LDP) of
     solution trajectories  as well
     as invariant 
     measures of \eqref{intr1}-\eqref{intr2}.

     The Freidlin-Wentzell  LDP   of
  stochastic   systems
  was     introduced in 
   \cite{fre1, fre2} 
   by  the  classical probabilistic method
   and further developed 
   in  
   \cite{bud1,  dup1}   
  by   the weak convergence  method.
  We   are  interested in the uniform 
  Freidlin-Wentzell
  LDP of solutions 
  with respect to initial data
  varying in a subset  of a phase space.
  Such uniform LDP of solutions   is 
  a key 
  to  determine      the 
    exit time
   and exit place of    solutions,
   and is also a  necessary
   tool to establish the LDP of invariant measures.

   Recently,
       the 
      uniform  LDP    of stochastic systems
      has been  investigated  in   
    \cite{bis1,  bud2, 
   cerr1, che1, dem1,  fre2, gau1,
    sal1, sal2, sal3, set1, sow1, wangSD2023,wangSD2026}
and the references therein.
In particular, 
in \cite{sal3}, the author
has established the
uniform   LDP of mild 
solutions of the 
stochastic reaction-diffusion
systems 
over all initial
data in the space of continuous
functions on a bounded domain
by the semigroup argument.
It seems that the
  result of \cite{sal3}
is the only one available in the literature
regarding the   LDP
of solutions  uniformly
over all  initial data in the entire
phase space
for 
reaction-diffusion systems
with unbounded multiplicative noise.
  Such a result is crucial
for proving the LDP of invariant measures
of the  stochastic  systems 
in that case.

Note that  
  the stochastic systems
in \cite{sal3} are defined in a
bounded spatial domain
where Sobolev
embeddings are compact
and 
the spectrum of the
Laplace operator
with standard boundary
conditions is discrete.
In  this case,
 the compactness
 of  Sobolev embeddings
provides an effective mechanism for converting spatial regularity estimates into
compactness of solutions. Together with appropriate temporal estimates, this
mechanism yields compactness of controlled trajectories and of action level
sets. It also supports the construction of compact sets needed in
exponential-tightness arguments for invariant measures. Moreover,
the discrete spectrum
of 
 the Laplacian 
allows
 approximation and noise assumptions to be formulated in terms of its
eigenfunctions. These properties play an important role in many bounded-domain
approaches.

On the other hand, 
there has been 
some  work  on the 
LDP of solutions for stochastic equations
 in unbounded domains.
In  \cite{wangJDE2023,wangSD2026},
the author  established 
the 
LDP of solutions 
uniformly
over bounded initial data
in $L^2(\R^n)$
for fractional
stochastic equations,
using spatial tail estimates to recover compactness.     More
recently,  in \cite{set1},
the author  proved  the  
LDP
of solutions 
uniformly over bounded initial data
in weighted function spaces
for  a  stochastic  equation 
on the real line.

As far as we are aware,
currently, 
there is no    result
on the LDP of solutions
which holds uniformly
for all initial data
for  the stochastic 
fractional  reaction-diffusion 
  equations
defined in  
unbounded domains.
The first 
goal of the present paper
is to fill this gap, and prove 
the uniform LDP of
weak  solutions
of \eqref{intr1}-\eqref{intr2} 
over all initial data
in $L^2(\R^n)$
by the variational argument,
where the spatial domain is
the entire space $\R^n$
and hence  is unbounded.

 Our  arguments
  to deal with the uniform  LDP
  of \eqref{intr1}-\eqref{intr2}
  in   $\R^n $
  are   substantially different
  from \cite{sal3}  in bounded domains
  due to the following issues:

  (i).  The embedding
\(H^\alpha(\mathbb R^n)\hookrightarrow L^2(\mathbb R^n)\) is not compact:
translations of a fixed nonzero function may remain uniformly bounded in
\(H^\alpha(\mathbb R^n)\) while having no strongly convergent subsequence in
\(L^2(\mathbb R^n)\). Thus, regularity estimates alone cannot prevent spatial
mass from escaping to infinity. 
The resulting
obstacles affect both finite-time and stationary large deviations. For
trajectories, one must recover compactness of controlled solutions despite the
noncompact spatial embedding. For invariant measures, one must additionally
control arbitrarily long time horizons and stationary fluctuations outside
neighborhoods of quasipotential level sets.

(ii). The Laplacian and its fractional
powers on \(L^2(\mathbb R^n)\) have continuous spectrum and do not possess the
discrete eigenfunction structure associated with compact resolvent. Therefore,  
the noise assumptions
cannot be    formulated in terms of 
the 
eigenfunctions
of those operators.  

(iii). The Lebesgue spaces $L^p(\R^n)$
are not nested in the sense that
$L^{p_1} (\R^n)$ is not included in
$L^{p_2} (\R^n)$ 
even if  $p_1> p_2$, and hence
the uniform estimates of solutions
in $L^{p} (\R^n)$ cannot be used to
control the solutions in $L^2(\R^n)$
even if $p>2$.

Our first main result establishes 
the  Freidlin-Wentzell uniform LDP for the
solutions of \eqref{intr1}-\eqref{intr2} in
the space
$$
\mathcal E_T
=C([0,T];
L^2(\R^n) 
)
\cap L^2(0,T;H^\alpha (\R^n))
\cap L^p(0,T;L^p(\mathbb R^n)),
$$
uniformly with respect to every initial state $u_0\in 
L^2(\R^n) 
$. This is stronger
than establishing a uniform LDP separately on each bounded subset of 
$L^2(\R^n) 
$,
because the small-noise estimates do not depend on a bound for the initial norm.
To the best of our knowledge, this is the first such result for fractional
reaction-diffusion equations on 
$\R^n$    with the potentially unbounded
multiplicative noise. The advance concerns the range of
uniformity and its consequences for invariant measures.

A major difficulty
 in obtaining 
 the uniform LDP
 of solutions of \eqref{intr1}-\eqref{intr2} 
for  all initial data 
 $u_0\in L^2(\R^n)$ is that conventional energy
estimates depend on 
$ \|u_0\|_{L^2(\R^n)}$. For multiplicative noise, the same dependence
enters estimates of the stochastic forcing through the solution-dependent
coefficient. An error  estimate 
depending on  the initial norm
cannot yield convergence uniformly over all 
$u_0\in  L^2(\R^n)$.
  Nor can bounded-domain
arguments based on finite spatial volume be transferred directly to
\(\mathbb R^n\). In particular, a uniform pointwise bound does not by itself
imply a uniform \(L^2(\mathbb R^n)\)-bound.

We overcome this difficulty by comparing controlled solutions with the same
control and different initial states. Let \(u_v(\cdot,u_0)\) denote the solution
of the skeleton equation obtained by replacing the stochastic forcing
in \eqref{intr1}
 with
\(\sigma(t,u_v)v\,dt\). A truncation argument with a time-dependent scalar
barrier, together with the monotonicity of the reaction term and the positivity
property of the fractional Dirichlet form, yields
\begin{equation}\label{intr3}
	\|u_v(t,u_0)-u_v(t,0)\|_{L^\infty(\mathbb R^n)}
	\le C_{T,N}\,t^{-1/(p-2)},
	\qquad 0<t\le T,
\end{equation}
uniformly over 
$u_0\in L^2(\R^n)$ and controls satisfying
$
\|v\|_{L^2(0,T;\ell^2)}\le N$,
where $C_{T,N}$ is a positive number
depending only on $T$ and $N$.

The comparison with the zero-initial-data solution is essential: the estimate
controls the effect of arbitrary initial data without requiring a global
\(L^\infty\)-bound for the externally forced reference solution.

The spatial localization of the noise then converts
\eqref{intr3} into a bound for its Hilbert--Schmidt norm along
controlled trajectories:
$$
	\sup_{u_0\in H}
	\sup_{\|v\|_{L^2(0,T;\ell^2)}\le N}
	\int_0^T
	\|\sigma(t,u_v(t,u_0))\|_{\mathcal L_2(\ell^2,H)}^2\,dt
	<\infty.
$$
Here the singularity near the initial time is bounded by \(t^{-2/(p-2)}\),
which is integrable precisely when \(p>4\). This explains the role of the
dissipativity exponent in our argument. Energy estimates for the difference
between controlled stochastic and deterministic solutions, followed by martingale
estimates, consequently give convergence in probability in \(\mathcal E_T\),
uniformly over all initial data and bounded-energy controls.

A separate compactness argument is required to show that the trajectory rate
functions are good. For each fixed initial state, spatial tail estimates are
combined with compactness on bounded spatial regions to obtain strong continuity
of the skeleton solution with respect to weakly convergent controls. This gives
compact action level sets in
$ \mathcal E_T$. 
It is not necessary to obtain
compactness of their union over all initial states, which would be incompatible
with unrestricted initial data at time zero. Combining this fixed-initial-state
compactness with the uniform stochastic convergence allows us to apply the
weak-convergence criterion for the Freidlin--Wentzell uniform LDP \cite{sal2}.
We do not identify this principle with a Dembo-Zeitouni uniform LDP on the
noncompact space 
$L^2(\R^n) $, since these notions are not equivalent in general.

Based on the uniform LDP of solutions, we then investigate
the LDP of invariant measures
of \eqref{intr1}-\eqref{intr2}.
The LDP of invariant measures of stochastic partial
 differential equations has been established 
   in \cite{brz1, cerr0, cerr2, klo1,mar1, mar2, sow2}
   when the spatial domains  are bounded.
   In the
    case of unbounded domains,
   the only work on the LDP of invariant measures
   was given in 
   \cite{wangJSP2024} 
   where 
    the fractional reaction-diffusion equation
   in $\R^n$
   with additive noise was  studied.
   In the present paper, we 
   will deal with
   the  LDP
   of invariant measures of  
   \eqref{intr1}-\eqref{intr2} 
   driven by unbounded multiplicative noise,
   which is  substantially different from
   the case of   additive  noise as treated in
   \cite{wangJSP2024}:
   
   (i).   The argument of \cite{wangJSP2024}
   is valid only for additive noise, and cannot be
   used to handle unbounded multiplicative noise.
   In particular, the weak exponential  tightness
   of solutions for additive noise  was employed  in 
    \cite{wangJSP2024} which is not available
    for unbounded multiplicative noise.

    (ii).  A  weighted Sobolev space  is   involved
    in establishing the exponential tightness of solutions
    in  \cite{wangJSP2024}
    which requires the noise coefficient also belongs
    to that wighted space. In the present  paper, 
    we remove the  restriction on weighted noise,
    and  do not use  the exponential tightness
    in order to avoid any weighted spaces
    in $\R^n$.
    
    (iii). The uniform Dembo-Zeitouni LDP    
    of solutions   plays a key role 
    in 
    \cite{wangJSP2024}
    for proving the LDP of invariant measures.
    In this paper, we use a new argument
    to deal with the multiplicative noise
    without invoking the      
   uniform Dembo-Zeitouni LDP    
    of solutions.

    When the linear damping
    parameter
    $\lambda$ in \eqref{intr1}
    is sufficiently large,
    we will show 
    the deterministic dynamics 
     has a unique globally attracting equilibrium
\(\bar u\), and the stochastic equation has a unique invariant measure
\(\mu^\varepsilon\) for sufficiently small \(\varepsilon\). 

Our  second  main result 
establishes
the LDP of  the family
$
\{\mu^\varepsilon\}$ 
of invariant measures
  in $L^2(\R^n)$   with good rate function:
$$
		J(z)=\lim_{\rho \to 0}\inf\biggl\{
		\frac12\int_0^T\|v(t)\|_{\ell^2}^2\,dt:
		T>0,\ v\in L^2(0,T;\ell^2),
		\ \ \|u_v(T,\bar u)-z\|_H<\rho
		\biggr\}.
$$
 To the best of our knowledge, this is
the first invariant-measure LDP for fractional stochastic reaction--diffusion
equations on \(\mathbb R^n\) driven by multiplicative noise. It extends the additive-noise setting of \cite{wangJSP2024}
through an argument that does not introduce weighted spaces.

Two ingredients distinguish the stationary proof. First, compactness of the
quasipotential sublevel sets follows from fixed-time controlled compactness and
contraction. Controlled trajectories starting from \(\bar u\) with bounded total
action remain bounded uniformly in their time horizon. Their endpoints over
arbitrarily long intervals can therefore be approximated uniformly by endpoints
obtained over one sufficiently long but fixed interval. Compactness of the latter
sets yields compactness of the quasipotential 
level sets.

Second, the stationary upper bound uses the trajectory LDP uniformly over all of
$L^2(\R^n)$. On each fixed time interval, this provides an exponentially small bound
for failure to follow a bounded-action controlled path, independently of the
state at the beginning of the interval. The Markov property permits this
estimate to be applied at every intermediate random state, without first
restricting that state to a bounded or compact set. Iterating the controlled
contraction over a number of intervals of order \(\varepsilon^{-1}\), and using
the stationary second-moment estimate, gives
\[
\mu^\varepsilon
\bigl\{z:\operatorname{dist}_
{L^2(\R^n)}
(z,J^s)\ge\delta\bigr\}
\le
N_\varepsilon e^{-(s-\eta)/\varepsilon}
+C_\delta\varepsilon\,\theta^{2N_\varepsilon},
\]
where 
$J^s=\{J\le s\}$, 
$0<\theta<1$, and
$ N_\varepsilon$ is of order $\varepsilon^{-1}\).
The second term is exponentially small, while the prefactor \(N_\varepsilon\)
has no effect on the large-deviation scale. Thus an ordinary second-moment
estimate, combined with contraction and uniformity over the entire phase space,
suffices for the stationary upper bound. No separate exponential-tightness
estimate is required in this proof, and the trajectory LDP is used only on a
fixed time interval.

The paper is organized as follows. Section 2 recalls the weak convergence
criterion for uniform large deviations. Section 3 specifies the coefficient
assumptions and recalls well-posedness of \eqref{intr1}-\eqref{intr2}. Section 4 develops the controlled
estimates and proves the trajectory LDP uniformly over all initial data.
Section 5 establishes compactness of the quasipotential 
level sets and proves
the LDP for invariant measures. Appendix A provides the truncation identity used
in deriving the initial-data-independent estimates.

 \section{Weak convergence theory 
 for uniform  LDP}
 \setcounter{equation}{0}
 
 We 
 recall  the weak convergence  theory
   on  the
  uniform LDP of  a family of random variables
  in  
  \cite{bud1, dup1, fre2,sal2}
  which will be used in later sections.

 Suppose 
   $\{W(t)\}_{t\in [0,T]} $
   is 
a cylindrical Wiener process 
in   $\ell^2$ on a stochastic basis
$(\Omega, \mathcal{F}, 
   \{ { \mathcal{F}} _t\} _{t\ge 0},  P )$,
   which implies that
   $W(t)$ takes values in 
   another separable Hilbert space
   $U$  and 
    the embedding
  $\ell^2 \hookrightarrow U$ is   Hilbert-Schmidt.
  Suppose 
  $(\caly,  d)$  
and
$(\caly_0, d_0)$
are complete separable metric spaces.
  Given 
  $\eps>0$ and $x\in \caly_0$,
  let 
   $\calg^\eps_x: C([0,T], U) \to \caly$ 
   be  a measurable map
   and 
 $
 X^\eps_x = \calg ^\eps_x (W)$.
  Let $\cala$
 be the space of  
 $ \ell^2$-valued  
 progressively measurable   processes
 $v$
 such that
 $\int_0^T \| v (t)\|_{\ell^2} ^2 dt<\infty$
 $P$-a.s. 
 and 
  $$
 \cala_N
 =\{v\in \cala: v  \in 
 \overline{B}_N(L^2(0,T; \ell^2)),
 \ P
  \text{-a.s.}
 \},
$$
   where 
   $\overline{B}_N(L^2(0,T; \ell^2))$
   is 
   the closed ball in 
   $L^2(0,T; \ell^2)$ 
   of radius $N$ and  center zero,
   which 
   is  equipped with 
     the weak topology.

   Given $x\in \caly_0$, let
  $\calg^0_x:    C([0,T], U) \to \caly$
  be a measurable map
  such that for each $N>0$ 
  and $\eta>0$,
 \be\label{convp}
 \lim_{\eps\to 0}
 \sup_{x\in \caly_0}
 \sup_{v\in 
 \cala_N}
 P
 \left (
 d
 \left (
 \calg_x^\eps \left (
 W+
 \eps^{-\frac 12}\int_0^\cdot
 v(s) ds
 \right ),
 \
 \calg_x^0
 \left (
 \int_0^\cdot v(s) ds
 \right )
 \right ) >\eta
  \right )
  =0.
 \ee
 We further  assume that
   for each  $N>0$
   and $x\in \caly_0$,
  \be\label{convp1}
  \left \{
  \calg^0_x  
  \left (\int_0^{\cdot} v(t) dt
  \right ):\
  v \in \overline{B}_N(L^2(0,T; \ell^2))
  \right \}
  \ \text{ 
  is a compact subset of }
  \  \caly.
  \ee
  Given  $x\in \caly_0$,
  let  $I_{T,x}: \caly \to [0, \infty]$ be a map defined   by,
  for  $\phi \in \caly$,
 \be\label{convp2}
  I_{T,x}(\phi)= \inf 
  \left \{
  {\frac 12} \int_0^T \| v(t)\|^2_{\ell^2} dt:\
   v\in L^2(0,T;  {\ell^2}),
   \  
  \calg^0_x \left (\int_0^\cdot v(t) dt\right ) =
  \phi
  \right \},
 \ee
 where 
  the infimum of the  empty set
   is    $+\infty$.
   
   By  \eqref{convp1}
   we find that
    $I_{T,x}$  
   is a good rate function  
       in the sense that 
     for each
 $s\in [0, \infty) $, the level set
 $I^s_{T,x}  =\{\phi \in \caly: I_{T,x}  (\phi) \le s\}$ is  compact.
 
 We need the following result
 on the uniform LDP of the family
 $\{X^\eps_x, \eps>0,
 x\in \caly_0 \}$.

   \begin{prop}\label{fwuldp}
   \cite[Theorem 2.13]{sal2}
   Under conditions 
    \eqref{convp} and \eqref{convp1},
      the  
    family $\{X^\eps_x, \eps>0,
 x\in \caly_0 \}$   satisfies 
 the  uniform Freidlin-Wentzell
 LDP in  $\caly$
 with  
 rate functions $I_{T,x}$  given
 by \eqref{convp2} 
 uniformly on $\caly_0$;
 more precisely,  
    $\{X^\eps_x, \eps>0,
 x\in \caly_0 \}$  satisfies
     the following two
 conditions:
 
 (i)  For every $s>0$  and $\eta>0$,
 $$
 \liminf_{\eps \to 0}
 \inf_{x\in \caly_0}
 \inf_{\phi
 \in I_{T,x} ^{s}}
 \left (
 \eps \ln P
 \left (
 d (X_x^\eps, \
 \phi  )
 < \eta
 \right )
 +I_{T,x} (\phi)
 \right )
 \ge 0.
 $$
 
 (ii)  For every $s_0>0$  and
 $\eta>0$,
 $$
 \limsup_{\eps \to 0}
 \sup_{x\in \caly_0}
 \sup_{0\le s\le s_0} 
 \left (
 \eps \ln P
 \left (
 {\rm dist} (X_x^\eps, \
 I_{T,x} ^s   )
\ge \eta
 \right ) +s  
 \right )
\le  0,
 $$
 \end{prop}
 where
 ${\rm dist}  (u, \
 E  )
 = \inf\limits_{\phi
 \in  E}
 d (u, \phi)$ for a subset $E$ of $\caly$.

   Proposition
   \ref{fwuldp}   will be applied to establish
   the uniform LDP
   of solutions of the stochastic equation
   \eqref{intr1}-\eqref{intr2}
   in 
   $C([0,T], L^2(\R^n))
   \bigcap L^2(0,T; H^\alpha (\R^n) )
   \bigcap L^p
   (0,T; L^p(\R^n))$
   over all initial data in $L^2(\R^n)$.

\section{Existence and uniqueness of solutions} 
\setcounter{equation}{0}

 In this section, we discuss assumptions
 on the nonlinear reaction and diffusion terms
 in \eqref{intr1}
 and recall the existence and uniqueness
 of solutions of   \eqref{intr1}-\eqref{intr2}.

Recall the    fractional Sobolev space
$H^ \alpha (\R^n)$  is    given by
  $$
 H^ \alpha (\R^n)
 =\left \{
 u\in L^2(\R^n): (-\Delta)^{\frac {\alpha}2}
 u \in L^2(\R^n) \right \},
  $$
  which is a
  Hilbert space with inner product
  $$
 (u, v)_{H^ \alpha (\R^n)}
 =
  \int_{\R^n} u(x)v(x)  dx 
  +
  \int_{\R^n}
  \left ( (-\Delta)^{{\frac  \alpha2}} u (x)
  \right )
  \left (
    (-\Delta)^{{\frac  \alpha2}} v (x)
    \right ) dx.
    $$ 
    The norm of 
   $H^ \alpha (\R^n)$
   is written as  
  $$
 \| u\|_{H^ \alpha (\R^n)}^2
 = \| u\|^2_{L^2(\R^n)}
 +   \| (-\Delta)^{{\frac  \alpha2}} u \|^2_{L^2(\R^n)}.
$$
It follows from 
    \cite{dine1}  that
   $$
 \| u\|^2_{H^ \alpha ({\R^n}) }
 = 
 \int_{\R^n} |u(x)|^2 dx 
  +
  {\frac {1}2}C(n,\alpha)
  \int_{\R^n} \int_{\R^n}
 {\frac {|u(x)- u(y)|^2} {|x-y|^{n+2 \alpha} }} dxdy , 
$$
and   
 $$
 (u, v)_{H^ \alpha (\R^n)}
 =
  \int_{\R^n} u(x)v(x)  dx 
  +  {\frac {1}2}C(n,\alpha)
  \int_{\R^n} \int_{\R^n}
 {\frac {(u(x)- u(y)) (v(x) -v(y))} {|x-y|^{n+2 \alpha} }} dxdy, 
 $$
for
  all $u, v \in H^ \alpha  ({\R^n})$
  with
    $C(n,\alpha) 
 =\frac{\alpha 4^\alpha\Gamma(\frac{n+2\alpha}{2})}
{\pi^\frac{n}{2}\Gamma(1-\alpha)}
$.

   In the sequel,
   we write
$H= 
  L^2(\R^n) $ 
  and
  $V= 
H^ \alpha (\R^n) $.
We use    
  $\| \cdot \|$  for the norm of $H$
  and  $(\cdot, \cdot)$ for the inner product of $H$.
Then
$V \hookrightarrow H  = H^* \hookrightarrow V^*$
where $H^*$   and $V^*$ are the dual spaces
of $H$  and $V$, respectively.
We use
$\call_2(\ell^2,H)$ for 
the  space of Hilbert-Schmidt
   operators 
   from $\ell^2$ to $H$
   with  norm
   $\| \cdot \|_{\call_2(\ell^2,H)}$.

     For the nonlinear term
      $F$ in \eqref{intr1},
we  assume  
$F: \R \times \o \times \R$
$\to \R$ is      continuous     
 such that  for all
$t, u, u_1, u_2  \in \R$   and  $x \in \o$, 
\begin{align}&
\label{f1}
F (t, x, u) u
\ge \lambda_1 |u|^p -\psi_1(t,x),
  \\
 &
\label{f2}
|F(t, x, u) |   \le
 \ \psi_2 (t,x)
  |u|^{p-1}  + \psi_3 (t,x),
  \\
 &
(F(t,x, u_1)
-F(t,x, u_2))(u_1-u_2) 
\label{f3}
\ge
\lambda_2 |u_1 -u_2|^p 
- \psi_4(t,x )|u_1-u_2|^2,  
\end{align}
where  $\lambda_1>0$,
  $\lambda_2>0$ and
   $ p>  4$ are constants, 
$\psi_1 \in L^1_{loc} (\R, L^1(\R^n))$,
$\psi_2, \psi_4  \in L_{loc}^\infty 
(\R, L^\infty(\R^n))$  and
$\psi_3\in 
  L_{loc}^{\frac p{p-1}}(\R, L^
  {\frac p{p-1}}
  (\R^n))$.

Let
        $\sigma:
      \R \times \R^n \times \R
      \to \ell^2$  given  by   
  \be\label{sig1}
     \sigma (t,x, u)
     = \sigma_1  (t,x)
     +   
       \kappa (x)   \sigma_{2} ( t,x, u)  
        , 
       \quad \forall \ t\in \R, \ x\in \R^n, \ u\in \R,
\ee
     where     
   $\kappa \in    L^2(\R^n) \bigcap L^\infty(\R^n)$,
   $\sigma_1\in L^\infty_{loc}
   (\R,  L^2(\R^n, \ell^2 ))$
   and 
    $ 
     \sigma_2 (t,x, u)
     =\left \{ 
        \sigma_{2,k} ( t,x, u)  \right \}_{k=1}^\infty
      $  such that
    $\sigma_{2,k}: \R \times \R^n \times \R
    \to \R$ is measurable
    and    for all 
   $t\in \R$,
   $x\in \R^n$  and $u_1, u_2 \in \R$,
    \be\label{sig2}
    \sigma_{2,k} (t,x, 0)=0
    \quad \text{and}
    \quad 
   |\sigma_{2,k} (t,x, u_1) -\sigma_{2,k} (t,x,u_2) |^2
   \le \alpha_k 
   |u_1  -u_2 |^2,
  \ee
   where  $\{\alpha_k\}_{k\in
   \mathbb{N}}$ is a 
  sequence of
   positive numbers 
   with
  $
  \sum_{k=1}^\infty  \alpha_k <\infty 
  $.

  Given $u\in H$  and $t\in \R$,  denote by 
  \be\label{sig23}
  \sigma (t,u) (v) (x)
  =\sum_{k=1}^\infty
  \left (
   \sigma_{1,k}  (t,x)
     +   
       \kappa (x)   \sigma_{2,k} ( t,x, u(x) )
       \right )
         v_k,
  \quad  
   v=\{v_k\}_{k=1}^\infty \in \ell^2,
   \ \ x\in \R^n.
\ee
 By \eqref{sig1}-\eqref{sig23} 
  we find  that 
  $\sigma (t, u): \ell^2\to H$ is a   Hilbert-Schmidt 
  operator such that 
      for $t\in \R$ and
    $u, u_1, u_2\in H$,
       \be\label{sig3}
    \|\sigma(t,  u_1)-
\sigma(t, u_2)\|^2_{\call_2 (\ell^2, H)}
\le
L_\sigma
\| u_1-u_2  \|^2 ,
    \ee
    and 
     \be\label{sig4}
    \|\sigma(t,  u)\|^2_{\call_2 (\ell^2, H)}
     \le
  2
  \|
  \sigma_1 (t)
  \|_{L^2(\R^n, \ell^2)}^2
  + 2
L_\sigma
   \| u \|^2 ,
    \ee
     where $L_\sigma
    =  
     \|\kappa \|_{L^\infty
  (\R^n)}^2\sum_{k=1}^\infty
  \alpha_k  $.

  Then  system \eqref{intr1}-\eqref{intr2}
  can be rewritten as 
  \be
  \label{sys1}
  du^{\eps} (t)
  + (-\Delta)^ \alpha  u^{\eps} (t)  dt
  +\lambda  u^{\eps} (t)  dt
    + F(t,x, u^{\eps}(t)) dt
  $$
  $$
  =  
    g(t,x)   dt
  +\sqrt{\eps} 
  \sum_{k=1}^\infty
  \left (
  \sigma_{1,k}  (t,x)
     +   
       \kappa (x)   \sigma_{2,k} ( t,x, u^{\eps}(t) ) 
       \right ) dW_k
   ,
  \ee 
  with   initial condition 
 \be\label{sys2}
 u^{\eps}( 0, x ) = u_0 (x),   \quad x\in \R^n,
 \ee 
  where $\{W_k\}_{k=1}^\infty$
  is a sequence of independent  real-valued
  standard 
  Wiener processes.

 The following  existence and uniqueness of solutions of
 \eqref{intr1}-\eqref{intr2}
 can be found in
 \cite{wangJDE2019}.

\begin{lem} \label{le31}
 Suppose 
 \eqref{f1}--\eqref{sig2}
  hold. Then   for every
    $T>0$,
  $\eps\in (0,1)$
  and  $u_0\in L^2(\Omega,\calf_0; H)$,
  system \eqref{intr1}-\eqref{intr2}
  has a 
   unique solution
  $u^\eps$  in the sense that
  $u^{\eps}$ is 
   a continuous $H$-valued $\calf_t$-adapted
   process such that 
$$
 u^{\eps}\in L^2(\Omega, C([0, T], H))
 \bigcap L^2(\Omega, L^2(0, T; V))
 \bigcap L^p(\Omega, L^p(0, T; L^p(\R^n))),
 $$
 and   for all  $t\in [0,T]$,  
$$
   u^{\eps}(t) 
 + \int_0^t
   (-\Delta )^{  \alpha  } u^{\eps}(s)  ds
 +
 \lambda \int_0^t u^{\eps}(s)  ds
 +
  \int_0^t  F(s, \cdot , u^{\eps}(s))    ds
 =u_0
   +\int_0^t  g(s)  ds
  + \sqrt{\eps} \int_0^t
  \sigma (s, u^{\eps}(s)) dW
$$
 in $ (V\bigcap L^p(\R^n) )^*$, $P$-almost surely. 
 Furthermore,  we have
 $$
\| u^{\eps}  \|_{L^2(\Omega, C([0, T], H) ) }^2
+
\| u ^{\eps}  \|_{L^2 (\Omega, L^2(0, T;  V ) )}^2
+
\| u ^{\eps} \| _{L^p(\Omega, L^p(0, T;  L^p(\R^n) ) )}^p
\le M_1
\left (1+ 
\| u_0 \|^2_{L^2(\Omega, H)} 
\right ),
$$
where $M_1>0$ is a constant
depending only on $T$.
 \end{lem}
  
     It follows from    Lemma \ref{le31}
     and the Yamada-Watanabe theorem 
        that for 
       $\eps\in (0,1)$ and
      $u_0\in H$, there  is 
      a 
       Borel-measurable map
     $\calg^\eps_{u_0}: C([0,T], U) \to  C([0, T], H )\bigcap  L^2(0, T;  V  )
     \bigcap  L^p(0, T;  L^p(\R^n)  )$
     such that the solution $u^\eps
     (\cdot, u_0) $ of
     \eqref{intr1}-\eqref{intr2} 
     is given by:
     $$
     u^\eps(\cdot, u_0) =\calg^\eps_{u_0} (W),
     \quad  \text{P-almost surely}.
     $$
     We 
     will
     establish  the   LDP of
     the family
     $\{\calg^\eps_{u_0}\}$ 
     in 
     $C([0, T], H )\bigcap  L^2(0, T;  V  )
     \bigcap  L^p(0, T;  L^p(\R^n)  )$
     uniformly for all $u_0\in H$.

   \section{Uniform LDP   of
   solution trajectories} 
\setcounter{equation}{0}

In this section, we prove our first main result
on the uniform 
Freidlin-Wentzell
 LDP 
 of solutions to  \eqref{intr1}-\eqref{intr2}
 over 
 all
  initial data in $H$
 which is  stated  below.

 \begin{thm}\label{main}
 	If \eqref{f1}-\eqref{sig2}
 	  hold true, then the solutions
 	of   
       \eqref{intr1}-\eqref{intr2}
       satisfy the
       Freidlin-Wentzell        LDP in
       $C([0, T], H )\bigcap  L^2(0, T;  V  )
       \bigcap  L^p(0, T;  L^p(\R^n)  )$,   as $\eps \to 0$, 
       uniformly with respect to all
       initial data in $H$
       with   good rate function
 given by \eqref{rate}.
\end{thm}

      As a first step to
       prove 
 Theorem \ref{main}, we
 need to consider 
 the  skeleton  equation
 with control     $v\in L^2(0,T; \ell^2)$:
   \be\label{contr1}
{\frac {d u_v (t)}{dt}}
+  (-\Delta)^\alpha u_v  (t)  
+  \lambda  u_v  (t)  +F(t, \cdot, u_v (t) )=  
  g (t)  
 +   \sigma(t, u_v (t))  
    v (t),
\ee
 with  initial data
 \be\label{contr2}
 u_v  (0)=u_0 \in H.
 \ee
 It follows from \cite{wangSD2026}
   that
under conditions     \eqref{f1}-\eqref{sig2},
   for 
   every 
  $v\in L^2(0,T; \ell^2)$ and $u_0\in H$,
\eqref{contr1}-\eqref{contr2}
 has a unique solution
 $u_v(\cdot, u_0)$ in 
  $ C([0,T], H)\bigcap L^2(0,T;V)
 \bigcap L^p(0,T; L^p(\R^n))$,
  such that    for all $t\in [0,T]$ 
  $$
   u_v (t)
   +\int_0^t
   (-\Delta)^\alpha u_v(s) ds
   +
   \lambda \int_0^t
     u_v(s) ds 
      +\int_0^t
   F(s,\cdot, u_v)ds
   =u_0 +\int_0^t 
   g(s) ds +\int_0^t
   \sigma(s, u_v(s)) v(s) ds,
$$
   in $\left (
   V\bigcap L^p(\R^n) \right )^*$.

   Next, we derive the uniform estimates
   of solutions of 
   \eqref{contr1}-\eqref{contr2}
  in   initial data in $H$.

 \begin{lem}\label{cos1}
 Suppose    \eqref{f1}-\eqref{sig2} are fulfilled,
  $T>0$  and 
  $N>0$.
  
  (i) There
   exists $M_2=M_2(T,  N)>0$ such that
  for all $v\in \overline{B}_N (L^2(0,T; \ell^2))$
  and $u_0\in H$,
  the solution
  $u_v$ of \eqref{contr1}-\eqref{contr2}
  satisfies  
\be\label{cos1 1}
   \sup_{t\in [0,T]}
   \| u_v (t, u_0  ) \|^2 
   +   \int_0^T \|  u_v  (s, u_0) \|^2 _V ds
    + 
   \int_0^T  \| u_v (s, u_0)  \|^p_{L^p(\R^n)}  ds
     \le M_2 (1+ \|u_0\|^2).
     \ee

     (ii)  There
   exists $M_3=M_3(T,  N)>0$ such that
  for all $v\in \overline{B}_N (L^2(0,T; \ell^2))$ and
  $u_{0,1}, u_{0,2} 
  \in H$,
  the solutions
  $u_v (\cdot, u_{0,1})$ 
  and $u_v(\cdot, u_{0,2}) $
  of \eqref{contr1}-\eqref{contr2}
  satisfy 
$$
   \sup_{t\in [0,T]}
   \| u_v (t, u_{0,1}   )
   - u_v(t, u_{0,2})
    \|^2 
   +   \int_0^T \|  
   u_v (s, u_{0,1}   )
   - u_v(s, u_{0,2})    \|^2 _V ds
    + 
   \int_0^T  \| 
   u_v (s, u_{0,1}   )
   - u_v(s, u_{0,2})     \|^p_{L^p(\R^n)}  ds
   $$
   \be\label{cos1 2}
     \le M_3 \| u_{0,1}-u_{0,2} \|^2 .
     \ee
     \end{lem}
   
   \begin{proof}
   (i). 
    By 
       \eqref{contr1}   we have
    $$
   {\frac 12} {\frac d{dt}}
   \| u_v (t) \|^2
   +\| (-\Delta)^{\frac {\alpha}2} u_v (t)\|^2
   +
   \lambda \|  u_v (t)\|^2    +\int_{\R^n} F(t,x, u_v(t))  u_v (t)  dx 
   $$
   \be\label{cos1 p1}
   =
     (g(t), u_v(t))
   +(\sigma (t, u_v(t)) v(t), u_v(t)).
  \ee
  By \eqref{f1}, \eqref{sig4}  and  \eqref{cos1 p1}
  we get
  $$
    {\frac d{dt}}
   \| u_v (t) \|^2
   + 2 \| (-\Delta)^{\frac {\alpha}2} u_v(t)\|^2
   + 2
   \lambda \|  u_v (t)\|^2   
       + 2\lambda_1
    \| u_v (t) \|^p_{L^p(\R^n)} 
   $$
   $$
   \le  \| u_v (t) \|^2 + \| g(t) \|^2
   + 2 \|\psi_1 (t) \|_{L^1(\R^n)}
      +
   \| \sigma (t, u_v (t))\|^2_{\call_2(l^2, H)} 
  + 
  \| v(t)\|^2_{l^2}\| u_v(t) \|^2
$$
\be
 \label{cos1 p2}
 \le  (1+2L_\sigma +  \| v (t)\|^2_{\ell^2}
 )  \| u_v (t) \|^2 + \| g(t) \|^2
   + 2 \|\psi_1 (t) \|_{L^1(\R^n)}
      +
   2\|\sigma_1(t) \|^2_{L^2(\R^n, \ell^2)}.
 \ee
 By \eqref{cos1 p2}
 and Gronwall's inequality  we  find  that
  for all $t\in [0, T]$
   and $\|v\|_{L^2
  (0,T; \ell^2)} \le N$,
\begin{align}&
   \| u_v (t) \|^2 
   + 2  \int_0^t \| (-\Delta)^{\frac {\alpha}2} u_v (s)\|^2 ds
    + 2  \lambda_1
   \int_0^t  \| u_v (s) \|^p_{L^p(\R^n)}  ds
 \nonumber\\
 &
     \le 
     e^{(1+2L_\sigma) T  +  N^2
        }\|u_{0}\|^2 
     +
      e^{(1+2L_\sigma) T  +  N^2
        }
          \int_0^T
        ( 
      \| g(s) \|^2
      + 2\|\sigma_1
  (s) \|^2_{L^2(\R^n, l^2)}
      + 2 \| \psi_1 (s) \|_{L^1(\R^n)}
       )
      ds,
     \end{align}
  which  gives \eqref{cos1 1}.

       (ii). 
       Let $u_1 (t) =u_v(t, u_{0,1})$
        and
         $u_2 (t) =u_v(t, u_{0,2})$
   for $t\in [0,T]$. 
          By
           \eqref{contr1}  we have
   $$
{\frac {d}{dt}} (u_1(t)-u_2(t))
+  (-\Delta)^\alpha (u_1(t)-u_2(t))
+  \lambda (u_1(t)-u_2(t))  
$$
\be\label{cos1  p3a}
=-( F(t, \cdot, u_1 (t) )
-F(t, \cdot, u_2 (t) ))+
    \sigma(t, u_1 (t))  
    v (t)
    -
    \sigma(t, u_2 (t))  
    v (t)  .
\ee
By \eqref{cos1 p3a} we obtain
       $$
     {\frac d{dt}}
     \| u_1(t) -u_2(t)\|^2
     +2  \| (-\Delta)^{\frac {\alpha}2}
     (u_1(t)-u_2(t))\|^2
     +2 \lambda  \| 
     u_1(t)-u_2(t)\|^2
     $$
     $$
    =-
     2\int_{\R^n}
      (F(t,x, u_1(t)) -F(t,x, u_2(t)))  (u_1(t)-u_2(t)) dx
     $$
     $$
     +2
     \left (
     \sigma (t, u_1(t))v(t)-
     \sigma (t, u_2(t))v(t),
     \  u_1(t)-u_2(t)
     \right ),
    $$
    and hence by 
    \eqref{f3}  and \eqref{sig3}  we get
\begin{align}&
     {\frac d{dt}}
     \| u_1(t) -u_2(t)\|^2
     +2  \| (-\Delta)^{\frac {\alpha}2}
     (u_1(t)-u_2(t))\|^2
     +2 \lambda \| u_1(t)-u_2(t)\|^2  
      \nonumber\\
     & + 2\lambda_2
     \int_{\R^n}
     | u_1(t,x)
     -
       u_2(t,x)|^p 
       dx
    \label{cos1 p3}
    \le  
    2(
     \|\psi_4(t)\|_{L^\infty (\R^n)}
      + L_\sigma^{\frac 12} 
     \|v(t)\| _{l^2}
     )
    \|u_1(t)-u_2(t) \|^2.
           \end{align}
           By Gronwall's inequality,
           we infer  from \eqref{cos1 p3} that
           for all $t\in [0,T]$,
      $$
     \| u_1(t) -u_2(t)\|^2
     +2 \int_0^t  \| (-\Delta)^{\frac {\alpha}2}
     (u_1(s)-u_2(s))\|^2 ds
     + 2\lambda_2
     \int_0^t 
     \| u_1(s)
     -
       u_2(s)\|^p _{L^p(\R^n)}
       ds
        $$
 $$
    \le  
    e^{2( T
    \|\psi_4\|_{L^\infty(0,T; L^\infty(\R^n))}
    + L_\sigma^{\frac 12}  T^{\frac 12} N
    )} 
    \|u_{0,1} -u_{0,2}  \|^2,
    $$
    which implies \eqref{cos1 2}.
        \end{proof}

   It follows from Lemma \ref{cos1} that
   the solutions  $u_v$ of \eqref{contr1}-\eqref{contr2}
   are  uniformly bounded in the space
   $C([0,T], H)$ for all
   $v\in \overline{B}_N (L^2(0,T; \ell^2))$
   if the initial value $u_0 =0$.
  In the sequel, we will  establish the uniform estimates 
   of $u_v$ 
   in $L^\infty (\R^n)$
    for arbitrary  initial data in $H$.
    To that end, we need the following
    lemma on the positive part of truncation of solutions.
    From now on, we write $s_+ =\max \{s, 0\} $
    for $s\in \R$.

   \begin{lem}\label{energy}
   Suppose $w\in
    C([0,T], H)\bigcap L^2(0,T;V)
 \bigcap L^p(0,T; L^p(\R^n))$
 is a solution of 
 $$
 {\frac {dw}{dt}}(t)
 = f_1 (t) + f_2 (t) \quad
 \text{in} \quad 
 (V\cap L^p(\R^n))^*
 $$
 in the sense of scalar distribution
 on (0,T), where 
 $f_1\in L^2(0,T; V^*)$
 and $f_2\in L^q(0,T; L^q(\R^n))$
 with ${\frac 1q}
 +{\frac 1p} =1$.
 Let $r: [0,T] \to (0,\infty)$ be an absolutely
 continuous function
 such that $r\in W^{1,1} (0,T)$
 and $ E(t)
 ={\frac 12} \| 
 (w(t) -r(t))_+\|^2$.
 Then $E\in W^{1,1} (0,T)$ with weak derivative:
 \be\label{energy1}
 {\frac {dE}{dt}} (t)
= (f_1(t), (w(t) -r(t))_+)_{(V^*,V)}
+\int_{\R^n}
f_2(t,x) (w(t,x) -r(t))_+    dx
-r^\prime (t)
\int_{\R^n}
(w(t,x) -r(t))_+  dx.
\ee
  \end{lem}
   
   The proof of Lemma \ref{energy}
   is given in Appendix  A.
   Next, we derive
     the uniform estimates 
   of solutions 
   of \eqref{contr1}-\eqref{contr2}
   in $L^\infty (\R^n)$
    for all  initial data in $H$
    by using Lemma \ref{energy}.

    \begin{lem}\label{cos2}
 If    \eqref{f1}-\eqref{sig2} are fulfilled,
  then for every
  $T>0$  and 
  $N>0$,
  there exists $M_4=M_4(T,  N)>0$ such that
  for all $v\in \overline{B}_N (L^2(0,T; \ell^2))$
  and for all  $u_0\in H$,
  the solution
  $u_v$ of \eqref{contr1}-\eqref{contr2}
  satisfies  for all $0<t \le T$,
  \be\label{cos2 1}
   \| u_v (t, u_0)-u_v (t, 0)  \|_{L^\infty
   (\R^n)}
   \le M_4  t^{- {\frac 1{p-2}}}.
   \ee
   \end{lem}
   
   \begin{proof}
     Note that
  \eqref{f3}
  implies that
  for $r, s, t \in \R$ and $x\in \R^n$,
  \be\label{cos2 p2}
  sgn (r-s) \left (
  F(t,x,r) -F(t,x, s)
  \right )
  \ge \lambda_2 |r-s|^{p-1}
  -\psi_4 (t,x) |r-s|,
  \ee
  By \eqref{sig23} we have
  for $t\in \R$ and $x\in \R^n$,
  \be\label{cos2 p3}
 \left  | \sigma(t, u_1 (t,x ))  
    v (t)
    -
    \sigma(t, u_2 (t,x ))  
    v (t) 
    \right  |
    \le L_\sigma^{\frac 12}
    \| v(t) \|_{\ell^2}
    \left |u_1 (t,x ) -u_2(t,x)\right |.
    \ee
    Denote by
   \be\label{cos2 p4a}
    \xi_v (t)
    =
     \int_0^t
            (\| \psi_4 (s)\|_{L^\infty (\R^n)}
            +L_\sigma^{\frac 12} \| v(s)\|_{\ell^2}) ds.
            \ee
            Then 
            for all 
            $t\in [0,T]$
            and
            $v\in \overline{B}_N (L^2(0,T; l^2))$
       we  have
       \be\label{cos2 p4} 
       0\le \xi_v (t) 
       \le 
       T \|\psi_4\| _{L^\infty
       (0,T; L^\infty
       (\R^n))}
       +L_\sigma^{\frac 12}
       N T^{\frac 12}.
       \ee

       Given $m>0$, define
       \be\label{cos2 p5}
       r_m (t)
       =e^{\xi_v (t) }
       \left(
       m^{-(p-2)}
       + \lambda_2 (p-2) \int_0^t
       e^{(p-2) \xi_v (s)} ds
       \right )^{
       -
       {\frac 1{p-2}}},
       \quad t\in [0,T].
       \ee
       By \eqref{cos2 p4} we have
       \be\label{cos2 p6}
         r_m (t)
       \le 
       e^{
       T \|\psi_4\| _{L^\infty
       (0,T; L^\infty
       (\R^n))}
       +L_\sigma^{\frac 12}
       N T^{\frac 12}       }
       \left(
        \lambda_2 (p-2)  t
       \right )^{
       -
       {\frac 1{p-2}}},
       \quad t\in (0,T].
       \ee
       On the other hand, by \eqref{cos2 p5}
       we see that
       $r_m$ is absolutely continuous on $[0,T]$
       and for  almost all $t \in (0,T)$,
       \be\label{cos2 p7}
       {\frac {d}{dt}}
       r_m (t)
       = 
       \left (\| \psi_4 (t)\|_{L^\infty (\R^n)}
            +L_\sigma^{\frac 12} \| v(t)\|_{\ell^2}
            \right )       r_m (t) 
       -\lambda_2 r_m^{p-1}  (t),
       \quad r_m (0) =m.
       \ee

       Let $u_1(t) = 
     u_{v} (t, u_{0})$,
     $u_2(t) = 
     u_{v} (t, 0)$
     and $w(t) = u_1 (t)
     -u_2 (t)$
   for $t\in [0,T]$. 
          By
           \eqref{contr1} we have
   \be\label{cos2 p8}
{\frac {d}{dt}}  w(t)
=-  (-\Delta)^\alpha  w(t)
-  \lambda  w(t)
-( F(t, \cdot, u_1 (t) )
-F(t, \cdot, u_2 (t) ))+
    (\sigma(t, u_1 (t))  
    -
    \sigma(t, u_2 (t)) 
    ) 
    v (t)  
\ee
in $(V\bigcap L^p(\R^n))^*$
in the sense of scalar distribution
on $(0,T)$.
By
\eqref{sig3}
and  \eqref{cos1 1}  we have
\be\label{cos2 p9}
-  (-\Delta)^\alpha  w
-  \lambda  w+
    (\sigma(\cdot, u_1 )  
    -
    \sigma(\cdot, u_2 ) 
    ) 
    v   \in L^2(0,T; V^*).
    \ee
    By \eqref{f2} and \eqref{cos1 1} we have
    \be\label{cos2 p10}
F(\cdot, \cdot, u_1  )
-F(\cdot, \cdot, u_2  )
   \in L^q(0,T; L^q(\R^n)).
\ee
It follows from
\eqref{cos2 p8}-\eqref{cos2 p10} and 
Lemma \ref{energy} that
$$
{\frac 12}
{\frac d{dt}}
\| (w(t) - r_m (t)) _+\|^2
+ (  (-\Delta)^\alpha  w(t),
(w(t) - r_m (t)) _+)_{(V^*, V)}
+  \lambda  ( w(t), (w(t) - r_m (t)) _+)
$$
$$
=
-\int_{R^n}
( F(t, x, u_1 (t) )
-F(t, x, u_2 (t) ) )
(w(t,x) - r_m (t)) _+ dx
$$
\be\label{cos2 p11}
+
    ((\sigma(t, u_1 (t))  
    -
    \sigma(t, u_2 (t)) 
    ) 
    v (t) , (w(t) - r_m (t)) _+)
     -r_m^\prime  (t) \int_{\R^n}
( w(t,x) - r_m (t)) _+ dx.
   \ee
     Since $s\to (s-r_m)_+$ is increasing,
     for the second term on the left-hand side
     of \eqref{cos2 p11} we have
     $$
     (  (-\Delta)^\alpha  w(t),
(w(t) - r_m (t)) _+)_{(V^*, V)}
$$
\be\label{cos2 p12}
={\frac 12} C(n,\alpha)
\int_{\R^n} \int_{\R^n}
{\frac 
{(w(t,x) -w(t,y))
(
(w(t,x) - r_m (t)) _+
-(w(t,y) - r_m (t)) _+) }
{|x-y|^{n+2 \alpha}}
}dxdy
\ge 0.
\ee
For the third term on the left-hand side of 
\eqref{cos2 p11} we have
\be\label{cos2 p13}
 \lambda  ( w(t), (w(t) - r_m (t)) _+)
=\lambda \int_{\{w(t,x)-r_m >0\}} w(t) (w(t,x) -r_m) dx \ge 0.
\ee
By \eqref{cos2 p2}, for the first term on the right-hand side 
of \eqref{cos2 p11} we obtain
$$
\int_{R^n}
( F(t, x, u_1 (t) )
-F(t, x, u_2 (t) ) )
(w(t,x) - r_m (t)) _+ dx
$$
$$
\ge 
\lambda_2
\int_{
\{ w(t,x)-r_m >0\}
}
w^{p-1}(t,x) (w(t,x) -r_m) dx
-
\|\psi_4 (t)\|_{L^\infty
(\R^n)}
\|( w(t) - r_m (t)) _+\|^2  
$$
\be\label{cos2 p13}
- \|\psi _4 (t) \|_{L^\infty
(\R^n)}
r_m (t) \int_{\R^n}
( w(t,x) - r_m (t)) _+ dx.
\ee
By \eqref{cos2 p3},
for the second term on the right-hand side of
\eqref{cos2 p11} we get
$$
 ((\sigma(t, u_1 (t))  
    -
    \sigma(t, u_2 (t)) 
    ) 
    v (t) , (w(t) - r_m (t)) _+)
    $$
    \be\label{cos2 p14}
    \le L_\sigma^{\frac 12}\|v(t)\|_{\ell^2}
    \| (w(t) -r_m (t) )_+ \|^2
    +  L_\sigma^{\frac 12}\|v(t)\|_{\ell^2} r_m (t)
    \int_{\R^n}
   (w(t,x) -r_m (t) )_+  dx.
   \ee
   
   By \eqref{cos2 p11}-\eqref{cos2 p14} we obtain
   $$
{\frac 12}
{\frac d{dt}}
\| (w(t) - r_m (t)) _+\|^2
+\lambda_2 
\int_{
\{ w(t,x)-r_m(t) >0\}
}
w^{p-1} (t,x) (w(t,x) -r_m) dx
$$
$$
\le 
(L_\sigma^{\frac 12}
\| v(t) \|_{\ell^2 } 
+
\|\psi_4 (t)\|_{L^\infty
(\R^n)} )
\|( w(t) - r_m (t)) _+\|^2  
$$
$$
+ 
(L_\sigma^{\frac 12}
\| v(t) \|_{\ell^2 } 
+
\|\psi_4 (t)\|_{L^\infty
(\R^n)} )
r_m (t) 
\int_{\R^n}
(w(t,x) - r_m (t)) _+ dx
$$
\be\label{cos2 p20}
     -r_m^\prime  (t) \int_{\R^n}
( w(t,x) - r_m (t)) _+ dx.
   \ee
   By \eqref{cos2 p7} and \eqref{cos2 p20} we get
   $$
   {\frac 12}
{\frac d{dt}}
\| (w(t) - r_m (t)) _+\|^2
+\lambda_2 
\int_{
\{ w(t,x)-r_m (t)>0\}
}
(w^{p-1} (t,x)-r_m^{p-1} (t) ) (w(t,x) -r_m (t) ) dx
$$
\be\label{cos2 p21}
\le 
(L_\sigma^{\frac 12}
\| v(t) \|_{\ell^2 } 
+
\|\psi_4 (t)\|_{L^\infty
(\R^n)} )
\|( w(t) - r_m (t)) _+\|^2 .
   \ee
      Since the second term on the left-hand side
      of \eqref{cos2 p21} is nonnegative,
      we
      obtain
       \be\label{cos2 p22}
{\frac d{dt}}
\| (w(t) - r_m (t)) _+\|^2
 \le  2
(L_\sigma^{\frac 12}
\| v(t) \|_{\ell^2 } 
+
\|\psi_4 (t)\|_{L^\infty
(\R^n)} )
\|( w(t) - r_m (t)) _+\|^2 .
   \ee

     We first consider the case
   $u_0\in H \bigcap L^\infty(\R^n)$.
   In this case,
   we  choose $m> \| u_0\|_{L^\infty(\R^n)}$.
   Then we have 
   $$
   w(0,x ) - r_m (0)  = u_0(x) -m \le 0
   \quad \text{for almost all } \ x\in \R^n,
   $$
   and hence by
      \eqref{cos2 p4a}
      and     Gronwall's inequality
   we obtain from
   \eqref{cos2 p22} that
   for all $t\in [0,T]$,
   $$
   \| (w(t) - r_m (t)) _+\|^2
 \le 
 e^{
 2 \int_0^t
(L_\sigma^{\frac 12}
\| v(s) \|_{\ell^2 } 
+
\|\psi_4 (s)\|_{L^\infty
(\R^n)} )ds}
\|( w(0) - r_m (0)) _+\|^2 
   $$
   \be\label{cos2 p23}
   =
   e^{2 \xi_v (t)} \|( w(0) - r_m (0)) _+\|^2 =0.
   \ee
   
   By \eqref{cos2 p23}
   we have
   for  every $t\in [0,T]$,
   $w(t) \le r_m (t)$ almost everywhere on
   $\R^n$.
   Similarly, by repeating the above procedure
   for   
   $(- w(t) - r_m (t)) _+   $ we can obtain
   $- w(t) \le r_m (t)$ almost everywhere on
   $\R^n$, and thus we get
     $\|w(t)\|_{L^\infty
     (\R^n)}
       \le r_m (t)$, which along with
       \eqref{cos2 p6} implies that
       for every $u_0 \in H\bigcap
       L^\infty (\R^n)$,
       \be\label{cos2 p30}
       \| w(t, u_0)\|_{L^\infty
       (\R^n)}
     \le     
       e^{
       T \|\psi_4\| _{L^\infty
       (0,T; L^\infty
       (\R^n))}
       +L_\sigma^{\frac 12}
       N T^{\frac 12}       }
       \left(
        \lambda_2 (p-2)  t
       \right )^{
       -
       {\frac 1{p-2}}},
       \quad  \forall \  t\in (0,T].
       \ee

       We now consider the case $u_0$ is an arbitrary element
       in $H$. 
       In this case,  we set
       $u_{0,m} (x) =  u_0 (x)$ if $| u_{0} (x) | \le m$;
       and  $u_{0,m} (x) =  0$  otherwise.
       By \eqref{cos2 p30} we have
        for every  $m\in \N$ and 
        $t\in (0,T]$,
       \be\label{cos2 p31}
       \| w(t, u_{0,m} ) \|_{L^\infty
       (\R^n)}
     \le     
       e^{
       T \|\psi_4\| _{L^\infty
       (0,T; L^\infty
       (\R^n))}
       +L_\sigma^{\frac 12}
       N T^{\frac 12}       }
       \left(
        \lambda_2 (p-2)  t
       \right )^{
       -
       {\frac 1{p-2}}} =\zeta (t).
       \ee
      By \eqref{cos1 2} we see that
       for all $t\in [0,T]$,
      \be\label{cos2 p32}
       \| u_v (t, u_{0,m})-u_v (t, u_0)
           \|^2
           \le c_1 \| u_{0,m} -u_0\|^2,
          \ee
           where $c_1>0$ is a number depending only
           on $T$ and $N$, but not on $m$.
           By \eqref{cos2 p32}  we find  that
           for all $t\in [0,T]$,
      \be\label{cos2 p33}
       \| w (t, u_{0,m})- w (t, u_0)
           \|^2
           \le c_1 \| u_{0,m} -u_0\|^2.
          \ee
          By \eqref{cos2 p31} and \eqref{cos2 p33} we infer that
          for every $u_0 \in H$ and  $t\in (0,T]$,
          $$
           \| ( |w (t, u_0)|  -\zeta (t))_+\|^2
           \le 
            \| w (t, u_{0})- w (t, u_{0,m} )
           \|^2 \to 0, \quad \text{as} \  m \to \infty,
           $$
           and hence
           for every $u_0 \in H$ and  $t\in (0,T]$,
          $$
           \| ( |w (t, u_0)|  -\zeta (t))_+\|^2=0,
           $$
           which implies that
           \be\label{cos2 p35}
           \| w(t, u_0)\|_{L^\infty
           (\R^n)}
           \le \zeta (t),
           \quad 
           \forall \ u_0\in H
           \ \text{and} \ 0<t\le T.
           \ee
           Then by \eqref{cos2 p31} and \eqref{cos2 p35}
           we obtain \eqref{cos2 1}.
          \end{proof}

 Next,
 we recall  
  the  
    continuity of
    solutions  of   \eqref{contr1}-\eqref{contr2}
      with respect to  controls 
 from  the weak topology of  $L^2(0,T; \ell^2)$
 to the norm topology of 
 $  C([0,T], H)\bigcap L^2(0,T;V)
 \bigcap L^p(0,T; L^p(\R^n))$, which is based
 on the uniform tail-ends estimates of solutions
 (see \cite{wangSD2026} for more details).

 \begin{lem}  {\rm (\cite[Lemma 4.4]{wangSD2026})}
 \label{wc}
 If
  \eqref{f1}-\eqref{sig2} hold and
       $v_n \to v$ weakly in 
  $L^2(0,T; \ell^2)$,
  then
  $ u_{v_n} \to u_v $ strongly
  in  
 $  C([0,T], H)\bigcap L^2(0,T;V)
 \bigcap L^p(0,T; L^p(\R^n))
 $. 
   \end{lem}
   \begin{proof}
   Under conditions \eqref{f1}-\eqref{sig2},
   the proof of this lemma is quite similar to 
   \cite[Lemma 4.4]{wangSD2026}
   with minor modifications, and hence is not
   repeated here.
   \end{proof}

    Given
    $u_0\in H$,
    define  an operator
   $\calg^0_{u_0}: C([0,T], U) \to C([0,T], H)
  \bigcap L^2(0,T; V)
  \bigcap
  L^p(0,T; L^p(\R^n))$  by:
   for  $\xi \in  C([0,T], U)$, 
  \be\label{calg0}
\calg^0_{u_0} (\xi)
=
\left \{
\begin{array}{ll}
u_v(\cdot, u_0),  & \text{ if } \xi= \int_0^\cdot v(t) dt
\ \text{ with } v\in L^2(0,T; \ell^2);\\
0, &  \text{ otherwise} ,
\end{array}
\right.
\ee
 where $u_v(\cdot, u_0)$ is the
 solution
 of \eqref{contr1}-\eqref{contr2}.
 Given   $T>0$ and  $u_0\in H$, 
  define  a mapping 
  $I_{T, u_0}:
    C([0,T], H)
  \bigcap L^2(0,T; V) \bigcap
  L^p(0,T; L^p(\R^n)) \to [0, +\infty]$ by: 
\be\label{rate}
 I_{T, u_0} (\phi)
 =\inf
 \left \{
 {\frac 12} \int_0^T \| v(s)\|_{\ell^2} ^2 ds:
 \ v\in L^2(0,T; \ell^2), \ u_v(\cdot, u_0)  =\phi
 \right \},
\ee
for every
  $\phi \in   C([0,T], H)
  \bigcap L^2(0,T; V) \bigcap
  L^p(0,T; L^p(\R^n))  $,  where $u_v(\cdot, u_0)$ is the
 solution
 of \eqref{contr1}-\eqref{contr2}.
By default,   $\inf \emptyset =+\infty$.
For each $u_0\in H$, $T>0$ and $s\ge0$,
denote by 
$$
 I_{T,u_0}^s
 =\left\{u_v(\cdot,u_0):v\in L^2(0,T;\ell^2),\quad
                       \frac12\int_0^T\|v(r)\|_{\ell^2}^2\,dr\le s\right\}.
    $$

 By Lemma \ref{wc}
 (see \cite{wangSD2026}), 
 one can verify that
 the rate function $I_{T,
 u_0}$ given by
 \eqref{rate} is a good rate function:
 for every $s\ge 0$,
 the set
 $I_{T,u_0}^s$ is  compact 
 in
 $
  C([0,T], H)
  \bigcap L^2(0,T; V) \bigcap
  L^p(0,T; L^p \R^n)) $,  and hence 
  compact in $C([0,T],H)$. 
   In particular, the endpoint set
$$
 \mathcal K_{T,s}
 =\left\{u_v(T, u_0 ):v\in L^2(0,T;\ell^2),\quad
                  \frac12\int_0^T\|v(r)\|_{\ell^2}^2\,dr\le s\right\}
$$
is compact in $H$.
Whenever
$I_{T,u_0}(\phi)<\infty$, the infimum in \eqref{rate} is attained.
 Moreover, 
   for every $N>0$, the set
\be\label{comg0}
 \left \{
\calg^0_{u_0}
\left (
\int_0^\cdot v(t) dt
\right ) : \  v\in {\overline{B}}_N(L^2(0,T; \ell^2))
\right \}
\ \ \text{is a compact subset
of  } \ 
\ee
 $ C([0,T], H)
  \bigcap L^2(0,T; V)\bigcap L^p
  (0,T; L^p(\R^n))$.

     The following lemma
    is concerned with
  the convergence of solutions
 of \eqref{intr1}-\eqref{intr2} as
 $\eps \to 0$.

 \begin{lem}\label{cp}
 	 If  \eqref{f1}-\eqref{sig2}
 	are fulfilled,  then for all
 	  $N>0$    and
 	$\eta>0$,  
 $$
 \sup_{u_0
 \in H
 	}
 \sup_{v\in\cala_N}
 P
 \left ( 
  \|	\calg^\eps_{u_0}
 	  (
 	W +\eps^{-\frac 12}\int_0^\cdot
 	v (t) dt
   )
 	-
 	\calg^0_{u_0}
   ( \int_0^\cdot
 	v (t) dt
 	  )\|
 	  _{
 	 C([0,T], H) \bigcap L^2(0,T; V)
 	 \bigcap L^p(0,T;
 	 L^p(\R^n) )
} >\eta 
\right )
 	  $$  
 	  $$\longrightarrow  0,
 	  \quad\text{ as } \eps \to 0.
 	  $$
  \end{lem}

\begin{proof}
Let
 $u_0\in H$, 
 $v\in \cala_N$ and 
 $u^\eps_v
=
 \calg^\eps_{u_0}
\left (
W +\eps^{-\frac 12}\int_0^\cdot
v (t) dt
\right )$.
Then by \eqref{sys1}-\eqref{sys2}  we have
 \be\label{geps 1}
  du_v^{\eps}  
  + (-\Delta)^ \alpha  u_v^{\eps}   dt
  +
  \lambda    u_v^{\eps}   dt
  + F(t, \cdot, u_v^{\eps} ) dt
  =  
    g(t)   dt
    +  \sigma (t,   u_v^{\eps} )    v dt
  +\sqrt{\eps} \sigma (t,   u_v^{\eps} )    {dW} , 
  \ee  
 with 
 $
 u^\eps_{v} (0)=u_0. $ 
   Let 
 $u_v=
 \calg^0_{u_0}
 \left (
 \int_0^\cdot
 v(t) dt
 \right )   $. 
By   
  \eqref{contr1} 
  and
  \eqref{geps 1}
  we find  that  
    $u_{v}^\eps -
     u_{v}$ satisfies
  \begin{align}
  &
 d  ( u_{v}^\eps -  u_{v}  )
 +     (-\Delta)^ \alpha ( u^\eps_{v}  -  u_{v  }
 ) dt 
 +\lambda
   ( u^\eps_{v}  -  u_{v  }
 ) dt 
 +
   ( F(t,\cdot, u^\eps_{v}  ) -  
   F(t, \cdot, u_{v}  ) ) dt   
 \nonumber\\
 &
 \label{cp p1}
  =
\left ( 
 \sigma(t, u^\eps_{v}  )  v
 -
  \sigma(t, u_{v}  )  v
\right )  dt
 + \sqrt{\eps} \sigma(t, u^\eps_{v}  )  
   dW  ,
\end{align}
with  
$u_{v}^\eps(0) -  u_{v} (0)=0$.
 By  \eqref{f3}
 and 
 \eqref{cp p1} 
 we get
 for all $t\in [0,T]$,
\begin{align}
&
 \|  u_{v}^\eps(t) -  u_{v}(t) \|^2
+2 
\int_0^t
\|  (-\Delta)^{\frac {\alpha}2}
( u^\eps_{v}(s)  -  u_{v} (s)
 ) \|^2  ds
 +2 \lambda 
 \int_0^t
 \|  
   u^\eps_{v}(s)  -  u_{v} (s)
   \|^2  ds
 \nonumber\\
 &
 +2\lambda_2 \int_0^t
 \int_{\R^n}
 \left  
 |  u^\eps_{v}(s,x)  -  u_{v} (s,x)
 \right |^p  dx ds
 \nonumber\\
 &
\le 2 \|\psi_4\|_{L^\infty(0,T;
L^\infty (\R^n))}
\int_0^t
\| u^\eps_{v}(s)  -  u_{v} (s)
\|^2 ds
 \nonumber\\
 &
 +
  2 
  \int_0^t
  \left ( 
 \sigma(s, u^\eps_{v} (s) )  v(s)
 -
  \sigma(s, u_{v} (s)  )  v (s),
  \  u^\eps_{v}(s)  -  u_{v}(s)
\right )  ds
 \nonumber\\
 &
   \label{cp p2}
+ \eps
\int_0^t
 \|   \sigma(s, u^\eps_{v} (s)  )\|
_{\call_2(\ell ^2, H)}^2  ds
+ 
2\sqrt{\eps} \int_0^t
\left (
u^\eps_{v}(s)  -  u_{v}(s),\
 \sigma(s, u^\eps_{v} (s)  ) 
   dW  \right ) .
\end{align}
     By  
  \eqref{sig3} 
    we have
\begin{align} \label{cp p3}
&
  2 
  \int_0^t
  \left |
  \left ( 
 \sigma(s, u^\eps_{v} (s) )  v(s)
 -
  \sigma(s, u_{v} (s)  )  v (s),
  \  u^\eps_{v}(s)  -  u_{v}(s)
\right ) 
\right |  ds 
  \le   2L_\sigma^{\frac 12}
  \int_0^t
   \| v (s) \|_{\ell ^2}
    \|
  u^\eps_{v}(s)  -  u_{v}(s)
  \| ^2  ds,
 \end{align}
and
$$
\int_0^t
\|   \sigma(s, u^\eps_{v} (s)  )\|
_{\call_2(\ell ^2, H)}^2  ds
$$
$$
\le
2 
\int_0^t
\|   \sigma(s, u^\eps_{v} (s)  )
-\sigma(s, u_{v} (s)  )
\|
_{\call_2(\ell ^2, H)}^2  ds
+
2 
\int_0^t
\|    
 \sigma(s, u_{v} (s)  )
\|
_{\call_2(\ell ^2, H)}^2  ds
$$
 \be  \label{cp p3a}
\le
2  L_\sigma 
\int_0^t
\|   u^\eps_{v} (s)   
-   u_{v} (s)   
\| ^2  ds
+
2 
\int_0^t
\|    
\sigma(s, u_{v} (s)  )
\|
_{\call_2(\ell ^2, H)}^2  ds
  . 
\ee 
It follows from
 \eqref{cp p2}-\eqref{cp p3a} 
 that
for all $t\in [0,T]$
and $\eps \in (0,1)$,
\begin{align}
&
 \|  u_{v}^\eps(t) -  u_{v}(t) \|^2
+2 
\int_0^t
\|  (-\Delta)^{\frac {\alpha}2}
( u^\eps_{v}(s)  -  u_{v} (s)
 ) \|^2  ds
  + 2 \lambda_2 \int_0^t
 \int_{\R^n}
  |u^\eps_{v}(s,x)
 -  u_{v}(s,x) |^p 
  dx ds
 \nonumber\\
 &
\le  2
\int_0^t
\left (
  \|\psi_4\|_{
 L^\infty (0,T; L^\infty(\R^n))}
+
L_\sigma^{\frac 12}
\| v(s)\|_{\ell^2}
+ L_\sigma  
\right )
\| u^\eps_{v}(s)  -  u_{v} (s)
\|^2 ds
 \nonumber\\
 &
   \label{cp p4}
    +
    2 \eps
    \int_0^t
    \|    
    \sigma(s, u_{v} (s)  )
    \|
    _{\call_2(\ell ^2, H)}^2  ds
 + 
2\sqrt{\eps} \int_0^t
\left (
u^\eps_{v}(s)  -  u_{v}(s),\
 \sigma(s, u^\eps_{v} (s)  ) 
   dW  \right ) .
\end{align}

Denote by
 \be\label{cp p5}
\zeta (t)
=
e^{ 
 -2
\int_0^t
\left (
\|\psi_4\|_{
	L^\infty (0,T; L^\infty(\R^n))}
+
L_\sigma^{\frac 12}
\| v(s)\|_{\ell^2}
+ L_\sigma  
\right ) ds
}.
\ee
 Then we have
\be\label{cp p6}
e^{-2
	(
\|\psi_4\|_{
	L^\infty (0,T; L^\infty(\R^n))}
+L_\sigma )T
- 2 
 L_\sigma^{\frac 12} T^{\frac 12} N}  
 \le
 \zeta (t) 
 \le 1 \quad   \text{for all } \ 
   t\in [0,T].
 \ee
  Applying It\^{o}'s 
  product rule
  to 
  $
  \zeta (t)
 	\|  u_{v}^\eps(t) -  u_{v}(t) \|^2$
	and using the
	differential  energy inequality
	underlying \eqref{cp p4}
	together with \eqref{cp p5},
 we obtain 
 for all $t\in [0,T]$,
 \begin{align}
 	& \zeta (t)
 	\|  u_{v}^\eps(t) -  u_{v}(t) \|^2
 	+2 
 	\int_0^t \zeta(s)
	\left (
 	\|  (-\Delta)^{\frac {\alpha}2}
 	( u^\eps_{v}(s)  -  u_{v} (s)
 	) \|^2 
 	+   \lambda_2 
 	\|  u^\eps_{v}(s)  -  u_{v} (s)
 	\|^p_{L^p(\R^n)} 
	\right )ds
   	\nonumber\\
 	&
 	\le 
 	\label{cp p7}
 	2 \eps
 	\int_0^t \zeta (s) 
 	\|    
 	\sigma(s, u_{v} (s)  )
 	\|
 	_{\call_2(\ell ^2, H)}^2  ds
 +
 2\sqrt{\eps} 
  \int_0^t
  \zeta (s) 
 \left (
 u^\eps_{v}(s)  -  u_{v}(s),\
 \sigma(s, u^\eps_{v} (s)  ) 
 dW  \right ).
  \end{align}
 Taking the expectation
 of \eqref{cp p7}
 (using a stopping time if necessary),
 by \eqref{cp p6}  we obtain 
 for all $t\in [0,T]$,
\begin{align}
&
\E \left (
 \|  u_{v}^\eps(t) -  u_{v}(t) \|^2
+2 
\int_0^t
\|  (-\Delta)^{\frac {\alpha}2}
( u^\eps_{v}(s)  -  u_{v} (s)
 ) \|^2  ds
 +  2 \lambda_2 \int_0^t
  \|  u^\eps_{v}(s)  -  u_{v} (s)
  \|^p_{L^p(\R^n)} ds
  \right )
 \nonumber\\
  &
 \le 
    \label{cp p10}
   c_1 \eps 
  \E \left ( 
  \int_0^T
   \|    
  \sigma(s, u_{v} (s)  )
  \|
  _{\call_2(\ell ^2, H)}^2  ds
   \right ),
\end{align}
where $c_1>0$ is a constant depending 
only on $T$ and $N$, but not on
$\eps$.

Next, we deal with the right-hand side
of
\eqref{cp p10}.
By \eqref{sig1}-\eqref{sig2},
\eqref{sig4} 
and  \eqref{cos1 1}   we have
for almost all $s\in (0,T)$,
$$
\|    
\sigma(s, u_{v} (s)  )
\|
_{\call_2(\ell ^2, H)}^2
\le
2
\|    
\sigma(s, u_{v} (s, u_0)  )
-
\sigma(s, u_{v} (s, 0)  )
\|
_{\call_2(\ell ^2, H)}^2
+
2\|
\sigma(s, u_{v} (s, 0 )  )
\|
_{\call_2(\ell ^2, H)}^2
$$
$$
\le 
2\sum_{k=1}^\infty \alpha_k
\int_{\R^n} \kappa ^2(x)
|u_{v} (s, u_0 ,x)
-u_{v} (s, 0, x) |
 ^2 dx
 + 4 \| \sigma_1 (s)\|^2
 _{L^2(\R^n, \ell^2)}
 +4L_\sigma\| u_{v} (s, 0 )\|^2
$$
$$
\le 
2\sum_{k=1}^\infty \alpha_k
 \| \kappa \|^2
\| u_{v} (s, u_0 )
-u_{v} (s, 0) \|^2_{L^\infty (\R^n)}
 + 4 \| \sigma_1 (s)\|^2
_{L^2(\R^n, \ell^2)}
+4L_\sigma\| u_{v} (s, 0 )\|^2
$$
\be\label{cp p11}
\le c_2 \| u_{v} (s, u_0 )
-u_{v} (s, 0) \|^2_{L^\infty (\R^n)}
+ c_2.
\ee
where $c_2>0$ is
a constant depending only on
$T$ and $N$, but not on $\eps$.
By   
 \eqref{cos2 1}
 and
 \eqref{cp p11}
   we infer that
for $p>4$ and  $t\in [0,T]$,
\be\label{cp p12}
\E \left (
\int_0^T
\|    
\sigma(s, u_{v} (s)  )
\|
_{\call_2(\ell ^2, H)}^2
ds
\right )
\le c_2 
\left (
T + c_3
\int_0^T
s^{-\frac 2{p-2}} ds
\right )
=
   c_2
   \left (T+   {\frac {(p-2)c_3}{p-4}}
T^{\frac {p-4}{p-2}}
\right ),
\ee 
where $c_3>0$ is a constant depending 
only on $T$ and $N$, but not on
$\eps$.
  It follows from \eqref{cp p10}-\eqref{cp p12}   that
  for 
   $t\in [0,T]$,
 \begin{align}\label{cp p13}
 	&
 	\E \left (
 	\|  u_{v}^\eps(t) -  u_{v}(t) \|^2
 	+2 
 	\int_0^t
 	\|  (-\Delta)^{\frac {\alpha}2}
 	( u^\eps_{v}(s)  -  u_{v} (s)
 	) \|^2  ds
 	+  2 \lambda_2 \int_0^t
 	\|  u^\eps_{v}(s)  -  u_{v} (s)
 	\|^p_{L^p(\R^n)} ds
 	\right )
 	\nonumber\\
 	&
    \le  c_1 c_2
 	\eps  \left (T+   {\frac {(p-2)c_3}{p-4}}
 	T^{\frac {p-4}{p-2}}
 	\right ).
 \end{align}
By \eqref{cp p3a}
and
 \eqref{cp p12}-\eqref{cp p13}
  we obtain 
  for all $\eps\in (0,1)$,
  $$
  \E \left (
   \int_0^T
 \|   \sigma(s, u^\eps_{v} (s)  )\|
 _{\call_2(\ell ^2, H)}^2  ds
 \right ) 
 $$ 
$$
 \le
 2  L_\sigma T
 \sup_{s\in [0,T]}
 \E \left (
  \|   u^\eps_{v} (s)   
 -   u_{v} (s)   
 \| ^2 
 \right )  
 +
 2 
 \int_0^T
 \E \left (
 \|    
 \sigma(s, u_{v} (s)  )
 \|
 _{\call_2(\ell ^2, H)}^2
 \right )   ds
 $$
  \be\label{cp p15}
 \le
 2 ( 
 L_\sigma Tc_1 +1)
   c_2
    \left (T+   {\frac {(p-2)c_3}{p-4}}
 T^{\frac {p-4}{p-2}}
 \right ).
 \ee 
 
 For the last term in \eqref{cp p7},
 by the BDG inequality
  we have
$$
 2\sqrt{\eps}
 \E \left (
 \sup_{t\in [0,T]}
 \left |
 \int_0^t
 \zeta (s) 
 \left (
 u^\eps_{v}(s)  -  u_{v}(s),\
 \sigma(s, u^\eps_{v} (s)  ) 
 dW  \right )\right | 
 \right )
 $$ 
 $$
 \le
   c_4 \sqrt{\eps}
 \E \left (  
 \left (
 \int_0^T \zeta^2(s)  \|
 u^\eps_{v}(s)  -  u_{v}(s)\|^2
 \| \sigma(s, u^\eps_{v} (s) )
 \|_{\mathcal{L}_2 (\ell^2, H)}^2
 ds  \right )^{\frac 12}
 \right )
 $$ 
\be\label{cp p16} 
 \le
 {\frac 12} \E \left (
  \sup_{s\in [0,T] } \zeta (s)
\|
 u^\eps_{v}(s)  -  u_{v}(s)\|^2
 \right )
 +{\frac 12} c_4^2 \eps
  \E \left (   
 \int_0^T  
 \| \sigma(s, u^\eps_{v} (s) )
 \|_{\mathcal{L}_2 (\ell^2, H)}^2
 ds  
 \right ),
 \ee
 where $c_4>0$ is a constant
 depending
 only on $T$ and $N$, 
 but not on
 $\eps$.
It follows from  \eqref{cp p7}, \eqref{cp p12}
 and   \eqref{cp p15}-\eqref{cp p16}
that 
 \be\label{cp p17} 
   \E \left (
 \sup_{s\in [0,T] }
 \|
 u^\eps_{v}(s)  -  u_{v}(s)\|^2
 \right )
 \le c_5 \eps,
 \ee
 where  
   $c_5>0$ is a constant
  depending
  only on $T$ and $N$, 
  but not on
  $\eps$.
Then by \eqref{cp p13}
and \eqref{cp p17}  we obtain
for all $u_0 \in H$,
$v\in \mathcal{A}_N $
and $\eps \in (0,1)$,
 \be\label{cp p18}
\E \left (
\sup_{t\in [0,T]}
\|  u_{v}^\eps(t) -  u_{v}(t) \|^2
+ 
\int_0^T
\|   
  u^\eps_{v}(s)  -  u_{v} (s)
 \|^2_V   ds
+   \int_0^T
\|  u^\eps_{v}(s)  -  u_{v} (s)
\|^p_{L^p(\R^n)} ds
\right )
\le c_6 \eps,
\ee 
where  
 $c_6>0$ is a constant
depending
only on $T$ and $N$, 
but not on
$\eps$.
By \eqref{cp p18}
and Chebyshev's inequality,
we infer that for every 
  $u_0 \in H$,
$v\in \mathcal{A}_N $,
 $\eps \in (0,1)$
     and
$\eta>0$,  
$$ 
P
\left ( 
\|
 u^\eps_{v}  -  u_{v} 
 \|   
_{
	C([0,T], H) \bigcap L^2(0,T; V)
	\bigcap L^p(0,T;
	L^p(\R^n) )
} >\eta 
\right )
$$  
$$
\le   
P
  ( 
\|
u^\eps_{v}  -  u_{v} 
\|   
_{
	C([0,T], H)  } >{\frac 13} \eta 
  )
  +  
   P
  ( 
  \|
  u^\eps_{v}  -  u_{v} 
  \|   
  _{
  	L^2(0,T; V)  } >{\frac 13} \eta 
  ) 
   +
   P
  ( 
  \|
  u^\eps_{v}  -  u_{v} 
  \|   
  _{
  	L^p(0,T; L^p(\R^n))  }
   >{\frac 13} \eta 
  ) 
$$  
$$
\le c_6\eps 
\left (
{\frac {18}{\eta^2}}
+ 
{\frac {3^p}{\eta^p}}
\right ),
$$
and thus for every $N\in\N$ and $\eta>0$, 
 $$
 \sup_{u_0\in H}
 \sup_{v\in 
 	\mathcal{A}_N}
 \left (
P
\left ( 
\|
u^\eps_{v}  -  u_{v} 
\|   
_{
	C([0,T], H) \bigcap L^2(0,T; V)
	\bigcap L^p(0,T;
	L^p(\R^n) )
} >\eta 
\right )
\right ) 
   \le c_6\eps 
\left (
{\frac {18}{\eta^2}}
+ 
{\frac {3^p}{\eta^p}}
\right ) \to 0,
$$  
as $\eps \to 0$,
 which completes the proof. 
 \end{proof}

We are now  in a position 
to complete the proof of
Theorem \ref{main} on the uniform LDP of solutions
of \eqref{intr1}-\eqref{intr2}
over all initial data in $H$.

\begin{proof}[Proof of Theorem \ref{main}]
Note that
\eqref{convp} follows from  Lemma \ref{cp},
and 
 \eqref{convp1} follows from
  \eqref{comg0}, 
  which along with  
  Proposition \ref{fwuldp} 
  yields 
   Theorem \ref{main}.
   \end{proof}

\section{Large deviations of invariant measures}
\setcounter{equation}{0}

In this section, we consider
the large deviations of invariant measures
of  the autonomous version of
 \eqref{intr1}-\eqref{intr2}:
\begin{equation}\label{aueq}
\begin{cases}
 du^\varepsilon(t)
 +\bigl(( -\Delta)^\alpha u^\varepsilon(t)+\lambda u^\varepsilon(t)
       +F(\cdot,u^\varepsilon(t))\bigr)\,dt
 =g\,dt+\sqrt\varepsilon\,\sigma(u^\varepsilon(t))\,dW(t),\\
 u^\varepsilon(0)=u_0,
\end{cases}
\end{equation}
where $g\in H$,  $F$ and $\sigma$ 
satisfy  the
 autonomous
versions of 
\eqref{f1}-\eqref{f3} and \eqref{sig1}-\eqref{sig2},
respectively.
In particular,  we now  assume
$
 \psi_1\in L^1(\mathbb R^n)$,
 $
 \psi_2,\psi_4\in L^\infty(\mathbb R^n)$
  $ \psi_3\in L^{p/(p-1)}(\mathbb R^n)$
  and
$\sigma_1\in L^2(\mathbb R^n;l^2)$.
In  addition,  in this section, we   assume  the
 strict dissipativity condition:
\begin{equation}\label{diss1}
 a:=\lambda-\|\psi_4\|_{L^\infty(\mathbb R^n)}>0.
\end{equation}
Condition \eqref{diss1} will be used 
to show 
the deterministic dynamics has a single globally attracting equilibrium.
No boundedness or nondegeneracy of the multiplicative noise is assumed.

\subsection{Energy estimates and the deterministic equilibrium}
In this subsection, we derive the 
uniform estimates of solutions
of the deterministic as well as the stochastic
equations  under
the strict dissipativeness condition
\eqref{diss1}.

\begin{lem}\label{lem5.1}
Suppose \eqref{f1}-\eqref{sig2}
 and \eqref{diss1} are fulfilled.
 If 
 $v\in L^2(0,T;\ell^2)$ and $u_{0,1},u_{0,2}\in H$,
 then
  the solutions
 $u_v(\cdot, u_{0,1})$
 and
 $u_v(\cdot, u_{0,2})$
  of \eqref{contr1}-\eqref{contr2}
  satisfy, for all $t\in [0,T]$,
\begin{equation}\label{lem5.1 1}
 \|u_v(t,u_{0,1})-u_v(t,u_{0,2})\|
 \le \exp\left(-at+\sqrt{L_\sigma}
                    \int_0^t\|v(r)\|_{\ell^2}\,dr\right)
        \|u_{0,1}-u_{0,2}\|.
\end{equation}
Consequently, if $\frac12\int_0^T\|v(r)\|_{\ell^2}^2\,dr\le s$, then
\begin{equation}\label{lem5.1 2}
 \|u_v(t,u_{0,1})-u_v(t,u_{0,2})\|
 \le\theta_s(t)\|u_{0,1}-u_{0,2}\|,
 \qquad
 \theta_s(t) = e^{
 -\frac a2t+\frac{L_\sigma}{a}s }.
\end{equation}
Moreover, the
 deterministic semigroup $S(t)u_0$
of \eqref{contr1}-\eqref{contr2}
with  zero control  has a unique equilibrium
$\bar u\in V\cap L^p(\mathbb R^n)$ which satisfies
\begin{equation}\label{lem5.1 3}
 \|S(t)u_0-\bar u\|\le e^{-at}\|u_0-\bar u\|,
 \qquad t\ge0,\quad u_0\in H.
\end{equation}
\end{lem}

\begin{proof}
The inequality \eqref{lem5.1 1} follows from
\eqref{cos1 p3}, \eqref{diss1} and Gronwall's inequality
immediately.
Note that
$$
 \sqrt{L_\sigma}\|v(r)\|_{l^2}
 \le\frac a2+\frac{L_\sigma}{2a}\|v(r)\|_{\ell^2}^2,
$$
which along with
\eqref{lem5.1 1} gives
\eqref{lem5.1 2}, where 
$\theta_s(t)\to 0$ as $t\to\infty$ for every fixed  $s \ge 0$.

If $v=0$, then by \eqref{cos1 p3} we find that the solutions
of the controlled equation satisfy
\be\label{lem5.1 p1}
 \|S(t)u_{0,1}-S(t)u_{0,2}\|
 \le e^{-at}\|u_{0,1}-u_{0,2}\|.
\ee
In particular, $S(1)$ is a strict contraction on $H$,
and  hence it has a unique 
 fixed point $\bar u$ in $H$.
Since
$S(1)S(t)\bar u=S(t)S(1)\bar u=S(t)\bar u$, 
the uniqueness of fixed points of  $S(1)$  implies
$S(t)\bar u=\bar u$ for every $t\ge0$,
and thus $\bar u$ is the unique  
equilibrium of the deterministic equation.
The contraction
\eqref{lem5.1 p1}  with $u_{0,2}=\bar u$ proves
\eqref{lem5.1 3}.
\end{proof}

From now on, we set
\begin{equation}\label{diss2}
 b:=\|\sigma(\bar u)\|_{L_2(\ell^2,H)}.
\end{equation}
Then by \eqref{sig3} and \eqref{diss2} we have
\begin{equation}\label{diss3}
 \|\sigma(u)\|_{L_2(\ell^2,H)}
 \le b+\sqrt{L_\sigma}\|u-\bar u\|,\qquad u\in H.
\end{equation}

\
\begin{lem} \label{lem5.2}
Suppose \eqref{f1}-\eqref{sig2}
and \eqref{diss1} are fulfilled.
For every $s\ge0$, define
\begin{equation}\label{Ms}
 M_s:=\sqrt{
 \frac{4b^2s}{a}}
                    e^{\frac{2L_\sigma s}{a}}.
\end{equation}
Then, for every $T>0$ and every $v\in L^2(0,T;\ell^2)$ with
$\frac12\int_0^T\|v(r)\|_{\ell^2}^2\,dr\le s$,
the solution $u_v(\cdot, \bar u)$ of 
\eqref{contr1}-\eqref{contr2} with initial value $\bar u$
satisfies
\begin{equation}\label{lem5.2 1}
 \sup_{0\le t\le T}\|u_v(t,\bar u)-\bar u\|\le M_s,
\end{equation}
and, for every $u_0\in H$,
\begin{equation}\label{lem5.2 2}
 \|u_v(t,u_0)-\bar u\|
 \le\theta_s(t)\|u_0-\bar u\|+M_s,\qquad 0\le t\le T.
\end{equation}
In particular, $M_s\to0$ as $s \to 0$.
\end{lem}

\begin{proof}
Put $y(t)=\|u_v(t,\bar u)-\bar u\|^2$. 
Since $\bar u$  is the equilibrium   of \eqref{contr1}-\eqref{contr2}
with $v=0$, 
by \eqref{contr1} and \eqref{diss3} we have 
\be\label{lem5.2 p1}
 y'+2ay\le
 2\sqrt{L_\sigma}\|v\|_{\ell^2}y
       +2b\|v\|_{\ell^2}\sqrt y.
\ee
Note that
$$
 2\sqrt{L_\sigma}\|v\|_{\ell^2}y
 \le\frac a2y+\frac{2L_\sigma}{a}\|v\|_{\ell^2}^2y,
 \qquad
 2b\|v\|_{\ell^2}\sqrt y
 \le\frac a2y+\frac{2b^2}{a}\|v\|_{\ell^2}^2,
$$
which along with \eqref{lem5.2 p1}
implies that
\begin{equation}\label{lem5.2 p2}
 y'+ay\le\frac{2L_\sigma}{a}\|v\|_{\ell^2}^2y
                         +\frac{2b^2}{a}\|v\|_{\ell^2}^2,
 \qquad y(0)=0.
\end{equation}
By Gronwall's inequality, we get
from \eqref{lem5.2 p2} that 
\begin{align*}
 y(t)
 \le\frac{2b^2}{a}\int_0^t
   \exp\left(-a(t-r)+\frac{2L_\sigma}{a}
                         \int_r^t\|v(\tau)\|_{\ell^2}^2\,d\tau\right)
                     \|v(r)\|_{\ell^2}^2\,dr
 \le\frac{4b^2s}{a}
 e^{
 \frac{4L_\sigma s}{a}},
\end{align*}
which proves
 \eqref{lem5.2 1}. Combining it with
\eqref{lem5.1 2}  we get
\eqref{lem5.2 2}. This completes the proof.
\end{proof}

\subsection{The transition semigroup and the Markov property}

In this subsection, 
  we assume $\eps
\in (0, \eps^*)
$ where $\eps^*$ is a small positive number
such that
\begin{equation}\label{snoise}
 2\varepsilon_*L_\sigma\le a.
\end{equation}

\begin{lem}\label{lem5.3}
Suppose \eqref{f1}-\eqref{sig2},
\eqref{diss1} and \eqref{snoise} are fulfilled.
Then
for 
 every  $\xi\in
 L^2(\Omega,\mathcal F_0;H) $,
 we have
\begin{equation}\label{lem5.3 1}
 \E
 \left (
 \|u^\varepsilon(t,\xi)-\bar u\|^2
 \right )
 \le 
 e^{-at}
 \E
 \left (
 \|\xi-\bar u\|^2
 \right )
              +\frac{2\varepsilon b^2}{a}(1-e^{-at}),\qquad t\ge0.
\end{equation}
If $\xi_1,\xi_2\in L^2(\Omega,\mathcal F_0;H)$, then
\begin{equation}\label{lem5.3 2}
 \E
 \left (
 \|u^\varepsilon(t,\xi_1)-u^\varepsilon(t,\xi_2)\|^2
 \right )
 \le e^{-at}
 \E
 \left (
 \|\xi_1-\xi_2\|^2\right ).
\end{equation}
\end{lem}

\begin{proof}
Let $z=u^\varepsilon(\cdot,\xi)-\bar u$. 
By  It\^o formula, \eqref{diss3} and \eqref{snoise} 
we get 
\begin{align*}
 d\|z\|^2
 &\le-2a\|z\|^2\,dt
    +\varepsilon\|\sigma(u^\varepsilon)\|_{L_2(\ell^2,H)}^2\,dt
    +2\sqrt\varepsilon\,(z,\sigma(u^\varepsilon)\,dW)\\
 &\le-(2a-2\varepsilon L_\sigma)\|z\|^2\,dt
    +2\varepsilon b^2\,dt
    +2\sqrt\varepsilon\,(z,\sigma(u^\varepsilon)\,dW)\\
    &\le -a \|z\|^2\,dt
    +2\varepsilon b^2\,dt
    +2\sqrt\varepsilon\,(z,\sigma(u^\varepsilon)\,dW),
    \end{align*}
    which along with Gronwall's lemma implies that
$$
 \E
 \left (\|z(t)\|^2 \right )
 \le  e^{-at}
\ E
\left (\|z(0)\|^2\right )+
 2\varepsilon b^2a^{-1}
 (1-e^{-at}) ,
$$
and thus 
 \eqref{lem5.3 1}
 is valid.

For the difference $w=u^\varepsilon(\cdot,\xi_1)-u^\varepsilon(\cdot,\xi_2)$,
by
the same calculation and \eqref{sig3},  we get
$$
 d\|w\|^2\le- a \|w\|^2\,dt
 +2\sqrt\varepsilon
       (w,(\sigma(u^\varepsilon(\cdot,\xi_1))
                  -\sigma(u^\varepsilon(\cdot,\xi_2)))\,dW),
$$
which together with Gronwall's inequality implies \eqref{lem5.3 2}.
\end{proof}

Let 
 $\mathcal B (H)$
 be the  Borel sigma-algebra
 of $H$ and $B_b(H)$ be the space of
 bounded Borel functions on $H$.
For $A\in\mathcal B(H)$ and $f\in B_b(H)$, define
$$
 P_t^\varepsilon(u_0,A):=P(u^\varepsilon(t,u_0)\in A),\qquad
 P_t^\varepsilon f(u_0):=\int_H f(z)P_t^\varepsilon(u_0,dz).
$$
Then by Lemma \ref{lem5.3}
and the standard argument,
one can verify that
the 
 operators $\{P_t^\varepsilon\}_{t\ge0}$ form a time-homogeneous Feller
Markov semigroup on $H$. 
More generally, if $h:H\times H\to\mathbb R$ is bounded and Borel
and $\xi\in L^2(\Omega, \calf_0; H)$,
then 
for all $\tau, t \ge 0$,
\begin{equation}\label{markov1}
 \begin{split}
 &\E\bigl[h(u^\varepsilon(\tau, \xi),u^\varepsilon(\tau+t, \xi))
                                     \mid\mathcal F_\tau\bigr]\\
 &\qquad=\int_H h(u^\varepsilon(\tau, \xi),z)
                      P_t^\varepsilon(u^\varepsilon(\tau, \xi),dz)
       \quad P\hbox{-almost surely}.
 \end{split}
\end{equation}

\subsection{Existence, uniqueness and concentration of invariant measures}

\begin{lem}\label{lem5.4}
Suppose \eqref{f1}-\eqref{sig2},
\eqref{diss1} and \eqref{snoise} are fulfilled.
Then for  each $0<\varepsilon < \varepsilon_*$, the semigroup
$\{P_t^\varepsilon\}_{t\ge 0}
$ has a unique invariant probability measure
$\mu^\varepsilon$ on $H$. Moreover,
\begin{equation}\label{leem5.4 1}
 \int_H\|u-\bar u\|^2\,\mu^\varepsilon(du)
          \le\frac{2\varepsilon b^2}{a},
\end{equation}
and
\begin{equation}\label{lem5.4 2}
 \mu^\varepsilon(H\setminus B_H(\bar u,r))
       \le\frac{2\varepsilon b^2}{ar^2},\qquad r>0.
\end{equation}
In particular, $\mu^\varepsilon \to \delta_{\bar u}$ 
weakly as
$\eps \to 0$.
\end{lem}

\begin{proof}
The existence of  an invariant measure
$\mu^\eps$  of 
$\{P_t^\varepsilon\}_{t\ge 0}$
can be proved by the argument of  
\cite{wangJDE2019}.

Let $\xi (u) =\| u-\bar u\|^2 $ for 
$u\in H$. Then by  \eqref{lem5.3 1}  with 
deterministic initial data we have
$$
P^\eps _t \xi  (u)
\le e^{-at} \xi (u) 
+   \frac{2\varepsilon b^2}{a}(1-e^{-at}).
 $$
 Given $R>0$, let $\psi_R (r) =\min\{ r, R\}$
 for $r\in \R$. By the concavity and monotonicity of $\psi_R$, we
 get
$$
 P^\eps _t  (\psi_R \circ \xi) 
 \le
 \psi_R ( P^\eps_t \xi)
\le 
\psi_R
\left (
e^{-at} \xi 
+   \frac{2\varepsilon b^2}{a}(1-e^{-at})
\right ).
$$
Then by the invariance of $\mu^\eps$ we have
$$
\int_H \left ( 
 \psi_R (\xi (u))
 -
\psi_R
\left (
e^{-at} \xi (u)
+   \frac{2\varepsilon b^2}{a}(1-e^{-at})
\right )
\right ) \mu^\eps (du)\le 0.
$$
Note that  the integrand 
is bounded below
by $-  \frac{2\varepsilon b^2}{a}(1-e^{-at})
$.
 Letting $R\to \infty$,  by Fatou's lemma, we get
$$
(1-e^{-at})
\int_H 
  \xi (u)
 \mu^\eps (du)\le 
 \frac{2\varepsilon b^2}{a}(1-e^{-at}),
 $$
 and hence
 $$
\int_H 
  \|u-\bar u\|^2 
 \mu^\eps (du)\le 
 \frac{2\varepsilon b^2}{a}.
 $$
 This  gives     \eqref{leem5.4 1},
from which \eqref{lem5.4 2} follows.
Note that  the  uniqueness of 
invariant measures follows
 from \eqref{lem5.3 2}
 and \eqref{leem5.4 1}.

On the other hand, for every $r>0$ and every 
$f: H\to \R$  bounded and continuous,
by \eqref{lem5.4 2} we get 
\[
 \left|\int_Hf\,d\mu^\varepsilon-f(\bar u)\right|
 \le\sup_{\|u-\bar u\|<r}|f(u)-f(\bar u)|
       +2\|f\|_\infty\frac{2\varepsilon b^2}{ar^2}.
\]
First letting  $\varepsilon
\to 0$ and then $r \to 0$, we  see that
$\mu^\eps \to \delta_{\bar u}$ weakly. 
\end{proof}

\subsection{The action and the quasipotential}

In this subsection, we regard  the rate function
 $I_{T,u_0}$  as a function on
$C([0,T],H)$.
Recall that
for every $T>0$, $u_0 \in H$   and $s\ge 0$,
the level set  $
 I^s_{T,u_0}$
 is compact in $C([0,T],H)$, and the   the endpoint set
$ \mathcal K_{T,s}$ is compact in $H$, where
$$ \mathcal K_{T,s} =\left\{u_v(T,\bar u):v\in L^2(0,T;\ell^2),\quad
                  \frac12\int_0^T\|v(r)\|_{\ell^2}^2\,dr\le s\right\}.
$$ 
 
Define the quasipotential by approximate terminal states:
for $u\in H$,
\begin{equation}\label{rate_IM}
 \begin{split}
 J(u)&:=\lim_{\rho\to  0}
 \inf\bigl\{I_{T,\bar u}(\phi):T>0,\ \phi\in C([0,T],H),
                                  \ \|\phi(T)-u\|<\rho\bigr\}\\
 &=\lim_{\rho \to 0}\inf\left\{
       \frac12\int_0^T\|v(r)\|_{\ell^2}^2\,dr:
       T>0,\ v\in L^2(0,T;l^2),\ \|u_v(T,\bar u)-u\|<\rho
                              \right\}.
 \end{split}
\end{equation}
The limit exists in $[0,+\infty]$ because the infima increase as
$\rho\downarrow0$.
 Put
$$
 J^s:=\{u\in H:J(u)\le s\},\qquad s\ge0,
$$
and define the finite-energy reachable set over arbitrary horizons by
\begin{equation}\label{reach}
 \mathcal R_s:=\bigcup_{T>0}\mathcal K_{T,s}.
\end{equation}

Next, we show $J$ is a good rate function on $H$.

\begin{lem}\label{lem5.5}
Suppose \eqref{f1}-\eqref{sig2},
\eqref{diss1} and \eqref{snoise} are fulfilled.
Then the function $J$ is a good rate function on $H$. More precisely,
$\overline{\mathcal R_s}$ is compact in $H$ for every finite $s\ge0$, and
\begin{equation}\label{lem5.5 1}
 J^s=\bigcap_{m=1}^\infty\overline{\mathcal R_{s+1/m}}.
\end{equation}
Moreover,
\begin{equation}\label{lem5.5 2}
 J(u)=0\quad \text{if and only if} \quad u=\bar u.
\end{equation}
\end{lem}

\begin{proof}
Fix $s\ge0$ and $T_0>0$. Consider an endpoint
$u_v(T,\bar u)\in\mathcal R_s$
with $\frac 12 \int_0^T
\| v(r) \|_{\ell^2}^2dr \le s$. 
If $0<T\le T_0$, 
then define
$\widetilde{v}: [0,T_0] \to \ell^2$ by
$$
\widetilde{v} (r)
=
\begin{cases}
0, \ & r \in [0, T_0 -T),\\
v(r-(T_0 -T)), \ & r\in [T_0 -T, T_0].
\end{cases}
$$
Then 
 $$
 \frac 12 \int_0^{T_0}
\| \widetilde{v} (r) \|^2_{\ell^2}dr
=
 \frac 12 \int_0^T
\| v(r) \|^2_{\ell^2}dr \le s,
$$
and hence
$ u_v(T, \bar u)
=u_{\widetilde{v}}
(T_0, \bar u)
\in \mathcal K_{T_0,s}$.

If $T>T_0$, define $\widetilde v(r)=v(T-T_0+r)$ for $0\le r\le T_0$.
Its energy is at most $s$. Restarting at time $T-T_0$ and applying
\eqref{lem5.1 2} and
\eqref{lem5.2 1} gives
$$
 \|u_v(T,\bar u)-u_{\widetilde v}(T_0,\bar u)\|
 \le\theta_s(T_0)\|u_v(T-T_0,\bar u)-\bar u\|
 \le\theta_s(T_0)M_s.
$$
Consequently,
$$
 \sup_{u\in\mathcal R_s}
           \operatorname{dist}_H(u,\mathcal K_{T_0,s})
     \le
     \theta_s(T_0)M_s\longrightarrow0\quad(T_0\to\infty).
$$
Given $r>0$, choose $T_0$ 
sufficiently large such that
$\theta_s(T_0)M_s<{\frac 12}r$.
Since 
$\mathcal K_{T_0,s}$
is compact, it has a finite cover of 
 open balls of
radius ${\frac 12} r$, and thus
  $\mathcal R_s$ 
  has a finite cover of open balls
  of radius of $r$.
  Since $r$ is arbitrary,
  we find that 
   $\mathcal R_s$ 
     is    totally bounded, and hence
   its closure is compact in $H$.

If $J(u)\le s$, then for every $m\ge1$ and every $\rho>0$ the definition
\eqref{rate_IM} provides an endpoint within distance $\rho$
of $u$ with energy less than $s+1/m$. Hence
$u\in\overline{\mathcal R_{s+1/m}}$ for every $m$.
Conversely, if $u$ belongs to every such closure, then for every $m$ and
$\rho>0$ there is an endpoint within distance $\rho$ of $u$ with energy
at most $s+1/m$. 
By taking the limit as  $\rho \to 0$ we get   $J(u)\le s+1/m$
for every $m\in \N$, and thus 
 $J(u)\le s$, which  proves
\eqref{lem5.5 1}.
Since $\overline{\mathcal R_{s+1/m}}$
is compact  for every $m$, by \eqref{lem5.5 1} we
conclude that $J^s$ is compact in $H$.

Next, we show \eqref{lem5.5 2}.
By definition of $J$, we have
 $J(\bar u)=0$.  On the other hand, 
 if $J(u)=0$,
 then by \eqref{rate_IM}  we find that
 for every $m\in \N$,
 there exists $T_m>0$
 and
 $v\in L^2(0,T_m; \ell^2)$
 such that
 ${\frac 12}
 \int_0^{T_m}  \| v(r)\|_{\ell^2} ^2 dr
 \le {\frac 1m}$
 and
 $\| u_v(T_m , \bar u) -u \|<{\frac 1m}$.
   By
\eqref{lem5.2 1}, we get
$$
\|u-\bar u\|
\le
\| u-u_v(T_m,\bar u)\|
+ \| u_v(T_m, \bar u)
-\bar u\|\le m^{-1}+M_{1/m}\longrightarrow0,
$$
which implies that  $u=\bar u$, as desired.
\end{proof}

We state explicitly the consequences of the all-initial-data LDP that
will be used. For every fixed $T>0$, $s>0$, $d>0$ and $\eta>0$, there
exists $\varepsilon_0=\varepsilon_0(T,s,d,\eta)>0$ such that
\begin{equation}\label{LDP s1}
 \sup_{u_0\in H}P\left(
  \operatorname{dist}_{C([0,T],H)}
       (u^\varepsilon(\cdot,u_0),I_{T,u_0}^s)\ge d\right)
 \le
 \exp\left(-\frac{s-\eta}{\varepsilon}\right),
 \qquad 
 0<\varepsilon<\varepsilon_0.
\end{equation}
After decreasing $\varepsilon_0$ if necessary, the same parameters give,
for every $u_0\in H$ and every control of energy at most $s$,
\begin{equation}\label{LDP s2}
 P\bigl(\|u^\varepsilon(\cdot,u_0)
 -u_v(\cdot,u_0)\|_{C([0,T],H)}<d\bigr)
 \ge
 \exp\left[-\frac1\varepsilon
           \left(\frac12\int_0^T\|v(r)\|_{\ell^2}^2\,dr+\eta\right)\right].
\end{equation}

\subsection{Large-deviation bounds for 
 invariant measures}

\begin{thm}\label{main1}
Suppose \eqref{f1}-\eqref{sig2},
\eqref{diss1} and \eqref{snoise} are 
fulfilled.
Then  the family
$\{\mu^\varepsilon:0<\varepsilon\ < \varepsilon_*\}$ satisfies the large
deviation principle on $H$ with  good rate
function $J$ given 
by \eqref{rate_IM}. Thus, for every open
$G\subseteq H$ and every closed $C\subseteq
 H$,
\begin{equation}\label{main1 1}
 \liminf_{\varepsilon
 \to 0}\varepsilon\log\mu^\varepsilon(G)
                 \ge-\inf_{u\in G}J(u),
\end{equation}
\begin{equation}\label{main1 2}
 \limsup_{\varepsilon\to
 0}\varepsilon\log\mu^\varepsilon(C)
                 \le-\inf_{u\in C}J(u).
\end{equation}
More precisely, for $u\in H$ and $\delta,\gamma>0$,
\begin{equation}\label{main1 3}
 \mu^\varepsilon(B_H(u,\delta))
       \ge\exp\left(-\frac{J(u)+\gamma}
       {\varepsilon}\right)
\end{equation}
for all sufficiently small $\varepsilon$, with $e^{-\infty}=0$; and for
$s>0$ and $\delta>0$,
\begin{equation}\label{main1 4}
 \limsup_{\varepsilon\downarrow0}
 \varepsilon\log
 \mu^\varepsilon\{u\in H:\operatorname{dist}_H
 (u,J^s)\ge\delta\}
          \le-s.
\end{equation}
\end{thm}

\begin{proof}
\emph{Lower bound.}
Fix $u\in H$ with $J(u)<\infty$ and $\delta,\gamma>0$.
By \eqref{rate_IM}, 
there exists  $T>0$ and $v\in L^2(0,T;l^2)$
such that
\begin{equation}\label{main1 p1}
 \frac12\int_0^T\|v(r)\|_{\ell^2}^2\,dr<J(u)+
 {\frac \gamma 4},
 \qquad
 \|u_v(T,\bar u)-u\|<{ \frac{\delta} 4}.
\end{equation}
Set $s=J(u)+\gamma/2$ and choose $r>0$ 
such that 
$\theta_s(T)r\le\delta/4$. Then,
by 
\eqref{lem5.1 2} and \eqref{main1 p1}  we have
\be\label{main1 p2}
 \|u_v(T,u_0)-u\|
 \le
  \|u_v(T,u_0)- u_v(T, \bar u)\| 
  + \|u_v(T, \bar u)-u\|  <\frac\delta2,
          \qquad u_0\in B_H(\bar u,r).
\ee
By \eqref{main1 p2} we get
for every 
$u_0\in B_H(\bar u,r)$,
$$
\left \{ \|u^\eps (\cdot, u_0) -u_v (\cdot, u_0)
\|_{C([0,T], H)}
<{\frac 12} \delta
\right \}
\subseteq
\left \{
\|
u^\eps (T, u_0) -u\|
< \delta
\right \},
$$
which along with
\eqref{LDP s2} 
for $d={\frac 12} \delta$
and $\eta ={\frac 14}
\gamma $
shows that
for every $u_0\in B_H(\bar u,r)$ and sufficiently small $\varepsilon$,
$$
 P_T^\varepsilon(u_0,B_H(u,\delta))
        \ge\exp\left(-\frac{J(u)+\gamma/2}{\varepsilon}\right).
$$
By the invariance of $\mu^\eps$ we have
\be\label{main1 p3}
 \mu^\varepsilon(B_H(u,\delta))
 =\int_H P_T^\varepsilon(u_0,B_H(u,\delta))\,\mu^\varepsilon(du_0)
 \ge\mu^\varepsilon(B_H(\bar u,r))
             \exp\left(-\frac{J(u)+\gamma/2}{\varepsilon}\right).
\ee
By \eqref{lem5.4 2}, 
$\mu^\varepsilon(B_H(\bar u,r)) \ge {\frac 12}$
 for all
sufficiently small $\varepsilon$. If also
$\varepsilon\le\gamma/(2\log2)$, then
$\frac12e^{\gamma/(2\varepsilon)}\ge1$,
which along with \eqref{main1 p3}
 proves
\eqref{main1 3}. When $J(u)=\infty$, that inequality is immediate.
If $G$ is an open set in $H$, then for every
$u\in G$, there exists $\delta>0$ such that
$B_H(u, \delta)  \in G$.
First taking
the logarithm 
of 
\eqref{main1 3}
and then   the supremum  over 
all  $u\in G$ we get
\eqref{main1 1}.

\medskip
\emph{Upper bound: a fixed-time estimate for all initial states.}
Fix $s>0$, $\delta>0$, and set
\begin{equation}\label{main1 p5}
 d=\frac\delta4,\qquad B=M_s+d.
\end{equation}
Since $M_s$ is independent of the time horizon and $\theta_s(T)\to0$,
by \eqref{main1 p5}
there exists 
$T>0$  such  that
 with $\theta=\theta_s(T)$,
\begin{equation}\label{main1 p6}
 0<\theta<1,\qquad \frac{\theta B}{1-\theta}\le\frac\delta4.
\end{equation}
 All of $T,d,B,\theta$ are now fixed
independently of $\varepsilon$ and of the number of time intervals used
below.

Let $u_0\in H$ and $\phi\in I_{T,u_0}^s$.  Then 
there exists 
 $v\in L^2(0,T;
\ell^2)$ with
${\frac 12} \int_0^T \| v(r)\|^2_{\ell^2} dr
\le s$ such that $\phi = u_v(\cdot, u_0)$.
By 
\eqref{rate_IM} and \eqref{reach} we have
$u_v(T, \bar u)\in  
\mathcal R_s\subseteq J^s$. 
Thus
\eqref{lem5.1 2} and
\eqref{lem5.2 2} imply that
\begin{equation}\label{main1 p7}
 \operatorname{dist}_H(\phi(T),J^s)
     \le\theta\|u_0-\bar u\|,
 \qquad
 \|\phi(T)-\bar u\|\le\theta\|u_0-\bar u\|+M_s.
\end{equation}
If the random path $u^\varepsilon(\cdot,u_0)$ has distance less than $d$
from $I_{T,u_0}^s$, 
then there exists a path $\phi\in I_{T, u_0}^s$ 
such that 
$\|u^\varepsilon-\phi\|_{C([0,T],H)}<d$. 
By the
 triangle inequality,
 \eqref{main1 p5}  and
\eqref{main1 p7}  we get 
that if
$\operatorname{dist}_{C([0,T], H)}
(u^\varepsilon(\cdot,u_0), I_{T, u_0}^s)
<d$,  then
 \begin{equation}\label{main1 p8}
 \begin{split}
 \|u^\varepsilon(T,u_0)-\bar u\|
                  &\le\theta\|u_0-\bar u\|+B,\\
 \operatorname{dist}_H(u^\varepsilon(T,u_0),J^s)
                  &\le\theta\|u_0-\bar u\|+d.
 \end{split}
\end{equation}

The distance function $D_s(z)=\operatorname{dist}_H(z,J^s)$ is
$1$-Lipschitz, since $J^s$ is nonempty. Define the open subset
$\Gamma\subset H\times H$ by
\begin{equation}\label{main1 p9}
 \begin{split}
 \Gamma=\{(u_0,z):\;&\|z-\bar u\|>\theta\|u_0-\bar u\|+B\ 
             \ \  \text{or }\  D_s(z)>\theta\|u_0-\bar u\|+d\},
 \end{split}
\end{equation}
and put
\begin{equation}\label{main1 p10}
 q_\varepsilon(u_0)
       =\int_H\mathbf1_\Gamma(u_0,z)P_T^\varepsilon(u_0,dz).
\end{equation}
By
\eqref{LDP s1} and  \eqref{main1 p8}-\eqref{main1 p10}
we obtain, 
for
any fixed $\eta\in(0,s)$,
\begin{equation}\label{main1 p12}
 q_\varepsilon(u_0)
 \le P\left(\operatorname{dist}_{C([0,T],H)}
           (u^\varepsilon(\cdot,u_0),I_{T,u_0}^s)\ge d\right)
 \le p_\varepsilon,
 \quad
 p_\varepsilon:=\exp\left(-\frac{s-\eta}{\varepsilon}\right),
\end{equation}
for every $u_0\in H$ and all $0<\varepsilon<\varepsilon_0(T,s,d,\eta)$.
 
\medskip
\emph{Iteration at stationary states.}
Let $\xi^\varepsilon\in L^2(\Omega, \calf_0; H) $
with law $\mu^\eps$ where $\mu^\eps$ is the invariant measure.
Then the solution $u^\eps (t,
\xi^\eps)$  has law $\mu^\eps$ for all $t\ge 0$.
For convenience, we write
$u^\eps(t) =u^\eps (t, \xi^\eps)$ and define
$$
 r_k=\|u^\varepsilon(kT)-\bar u\|,\qquad k\ge0.
$$
For $k\ge1$, let $\mathcal G_k$ be the event 
given by
\begin{equation}\label{main1 p13}
\mathcal G_k
=\left \{
 r_k\le\theta r_{k-1}+B,
 \quad
 D_s(u^\varepsilon(kT))\le\theta r_{k-1}+d
 \right \}.
\end{equation}
By \eqref{main1 p13} we have
\be\label{main1 p14}
 \mathbf1_{\mathcal G_k^c}
 =\mathbf1_\Gamma(u^\varepsilon((k-1)T),u^\varepsilon(kT)).
 \ee
Applying  \eqref{markov1} with $h=\mathbf1_\Gamma$,
by \eqref{main1 p10}, \eqref{main1 p12} and \eqref{main1 p14}
we get
\begin{equation}\label{main1 p20}
 \begin{split}
 P(\mathcal G_k^c\mid\mathcal F_{(k-1)T})
 =q_\varepsilon(u^\varepsilon((k-1)T))
 \le p_\varepsilon,
 \qquad P\hbox{-almost surely}.
 \end{split}
\end{equation}
The intermediate state is not restricted to a bounded or compact set.
The uniform estimate for every deterministic state in
\eqref{main1 p12} is exactly what permits this substitution. For
fixed $\varepsilon$, 
\eqref{main1 p20}
 may be taken to hold
simultaneously for all $k\ge1$ by discarding a countable union of null
sets.

Taking expectations in \eqref{main1 p20},
we obtain  for every integer $N\ge1$,
\begin{equation}\label{main1 p21}
 P\left(\bigcup_{k=1}^N\mathcal G_k^c\right)
 \le\sum_{k=1}^NE\bigl[P(\mathcal G_k^c\mid\mathcal F_{(k-1)T})\bigr]
 \le Np_\varepsilon.
\end{equation}
 On
$\bigcap_{k=1}^N\mathcal G_k$, iteration of 
$r_k\le\theta r_{k-1}+B$ for $1\le k \le N-1$  yields
$$
 r_{N-1}\le\theta^{N-1}r_0
             +B\sum_{j=0}^{N-2}\theta^j
 =\theta^{N-1}r_0+B\frac{1-\theta^{N-1}}{1-\theta},
$$
with the sum empty when $N=1$,
which along with 
\eqref{main1 p5}-\eqref{main1 p6}
and \eqref{main1 p13} implies that 
\begin{equation}\label{main1 p22}
 D_s(u^\varepsilon(NT))
 \le\theta^Nr_0+\frac{\theta B}{1-\theta}+d
 \le\theta^Nr_0+\frac\delta2.
\end{equation}
It follows from \eqref{main1 p22} that
if 
 $r_0\le\delta/(4\theta^N)$, 
 then
 on
$\bigcap_{k=1}^N\mathcal G_k$,
$$
 D_s(u^\varepsilon(NT))
 \le {\frac 14}\delta + {\frac 12}\delta 
  <\delta,
$$
which implies that
$$
 \{D_s(u^\varepsilon(NT))\ge\delta\}
 \subset\bigcup_{k=1}^N\mathcal G_k^c
       \ \cup\ \{r_0>\delta/(4\theta^N)\}.
$$ 
Stationarity, \eqref{main1 p21} and
\eqref{leem5.4 1} consequently imply
\begin{align}\label{main1 p23}
 &\mu^\varepsilon\{u:D_s(u)\ge\delta\}
 = P 
 \{D_s( u^\eps (NT) )\ge\delta\} 
 \notag\\
  &\le N\exp\left(-\frac{s-\eta}{\varepsilon}\right)
       +\mu^\varepsilon\{u:\|u-\bar u\|>\delta/(4\theta^N)\}\notag\\
 &\le N\exp\left(-\frac{s-\eta}{\varepsilon}\right)
       +\frac{16\theta^{2N}}{\delta^2}
                    \int_H\|u-\bar u\|^2\,\mu^\varepsilon(du)\notag\\
 &\le N\exp\left(-\frac{s-\eta}{\varepsilon}\right)
       +\frac{32\varepsilon b^2}{a\delta^2}\theta^{2N}.
\end{align}

\medskip
\emph{Choice of the number of intervals.}
Since $0<\theta<1$, let
\begin{equation}\label{main1 p24}
 c_0=\frac{s+1}{2|\log\theta|}>0,
 \qquad
 N_\varepsilon=\left\lceil\frac{c_0}{\varepsilon}\right\rceil,
\end{equation}
where  $\lceil s\rceil$ is the smallest integer
which is greater 
than or equal to $s$.
Then 
by \eqref{main1 p24}
we have
 $c_0/\varepsilon\le N_\varepsilon<c_0/\varepsilon+1$, and
\begin{equation}\label{main1 p30}
 \theta^{2N_\varepsilon}
   =\exp(-2|\log\theta|N_\varepsilon)
   \le\exp\left(-\frac{s+1}{\varepsilon}\right).
\end{equation}
For $0<\varepsilon\le1$,
\begin{equation}\label{main1 p31}
 0\le\varepsilon\log N_\varepsilon
 \le\varepsilon\log(c_0+1)+\varepsilon\log(1/\varepsilon)
 \longrightarrow0, \quad \text{as} \  \eps \to 0.
\end{equation}
 
Set $K_\delta=32b^2/(a\delta^2)$. 
By \eqref{main1 p23}-\eqref{main1 p30}
we get
\begin{align}\label{main1 p32}
 \mu^\varepsilon\{u:D_s(u)\ge\delta\}
 &\le N_\varepsilon e^{-(s-\eta)/\varepsilon}
                  +K_\delta\varepsilon e^{-(s+1)/\varepsilon}\notag\\
 &=e^{-(s-\eta)/\varepsilon}
      \left[N_\varepsilon+
              K_\delta\varepsilon e^{-(1+\eta)/\varepsilon}\right]\notag\\
 &\le e^{-(s-\eta)/\varepsilon}
                  N_\varepsilon(1+K_\delta\varepsilon).
\end{align}
The last inequality uses $N_\varepsilon\ge1$ and
$e^{-(1+\eta)/\varepsilon}\le1$. Taking logarithms, 
and using $\log(1+x)\le x$ for
$x\ge0$, we obtain from \eqref{main1 p32} that
\begin{equation}\label{main1 p33}
 \varepsilon\log\mu^\varepsilon\{u:D_s(u)\ge\delta\}
 \le-(s-\eta)+\varepsilon\log N_\varepsilon+K_\delta\varepsilon^2.
\end{equation}
 First taking the limit of \eqref{main1 p33} as
$\varepsilon \to 0$ and then $\eta \to 0$,
by \eqref{main1 p31} we obtain 
\eqref{main1 4}. Throughout this argument the LDP is applied only
on the fixed interval $[0,T]$; the interval of length $N_\varepsilon T$
is handled by the Markov property, not by an LDP
on a growing time horizon.

\medskip
\emph{Closed-set upper bound.}
Let $C$ be a  closed nonempty  subset of $H$,
and  $c=\inf_{u\in C}J(u)$. If $c=0$, then
\eqref{main1 2} follows from $\mu^\varepsilon(C)\le1$.
If $c>0$, fix any finite $s$ with $0<s<c$. The sets $C$ and $J^s$ are
disjoint, and $J^s$ is nonempty and compact. Define
\begin{equation}\label{main1 p40}
 \delta_s:
       =\inf_{\substack{u\in C\\z\in J^s}}\|u-z\|
       =\inf_{u\in C}\operatorname{dist}_H(u,J^s).
\end{equation}
This number is strictly positive since $C$ is closed and
$J^s$ is compact.
By  \eqref{main1 p40}, we see that  every $u\in C$
has distance at least $\delta_s$ from $J^s$, and thus
$$
 C\subseteq \{u:\operatorname{dist}_H(u,J^s)\ge\delta_s\},\qquad
 \mu^\varepsilon(C)
 \le\mu^\varepsilon\{u:\operatorname{dist}_H(u,J^s)\ge\delta_s\}.
$$
Apply \eqref{main1 4} with this fixed positive $\delta_s$ to obtain
$$
 \limsup_{\varepsilon\downarrow0}\varepsilon\log\mu^\varepsilon(C)
        \le-s.
$$
Letting
 $s \to c$ if $c<\infty$, and $s\to\infty$ if $c=\infty$,
 we get
 \eqref{main1 2}. 
 This completes the proof.
 \end{proof}

\begin{rem}\label{rem5.7}
The proof uses neither a compact embedding $V\hookrightarrow H$ on
$\mathbb R^n$ nor an assumed exponential-tightness estimate for
$\mu^\varepsilon$. Compactness of the quasipotential levels follows from
fixed-time controlled compactness. The stationary upper bound uses only
the ordinary second moment \eqref{leem5.4 1} together with
the uniform estimate \eqref{LDP s1} over all of $H$.
A bounded-initial-data estimate alone would not justify
\eqref{main1 p20} at unrestricted intermediate states without
additional information on excursions outside that bounded set. No
 Dembo-Zeitouni uniform LDP on a noncompact initial-state
set is used.
\end{rem}

\begin{rem}\label{rem5.8}
If $g=0$ and $F(x,0)=0$,
then the  uniqueness of solutions  
implies  $\bar u=0$. More generally,
if $\sigma(\bar u)=0$, then $b=0$  by
\eqref{diss2}, which along with
\eqref{Ms} and \eqref{lem5.2 1}
\ shows that every controlled
trajectory from $\bar u$ is constant. In this case
$J(\bar u)=0$, $J(u)=+\infty$ for $u\ne\bar u$, and
\eqref{leem5.4 1} gives $\mu^\varepsilon=\delta_{\bar u}$.
The theorem includes this degenerate case without modification.
\end{rem}

\appendix
  \section*{Appendix A}
  \setcounter{equation}{0}
  \renewcommand{\theequation}{A.\arabic{equation}}
\setcounter{thm}{0}
\renewcommand{\thethm}{A.\arabic{thm}}

  In this appendix, we prove
   Lemma \ref{energy}  which is crucial for deriving the uniform
   estimates of solutions of the controlled equation
   in $L^\infty (\R^n)$
   for all positive time.
   To that end, we first recall the translation continuity
   of functions in  a Bochner space.
   
   \begin{lem} \label{cbi}
   If $X$ is a Banach space  and $\zeta
   \in L^s(\R, X)$ with $1\le s<\infty$, then
   \be\label{cbi 1}
   \lim_{\tau \to
   0 }
   \int_{\R}
   \|
   \zeta (t \pm \tau) -\zeta (t)\|_X^s dt
   =0.
   \ee
   \end{lem}
   \begin{proof}
   The result follows immediately from
   \cite[Remark 3.2]{sim1}  along  with a truncation argument
   to approximate
   $\zeta$ by $ 1_{[-R,R]} \zeta$
   in $L^s (\R, X)$ as $R\to \infty$.
    \end{proof}
    
    Note that \eqref{cbi 1} implies that
    \be\label{cbi 2}
     \lim_{\tau  \to
   0 } 
   \| \zeta (\cdot \pm \tau)
   -\zeta \|_{L^s(\R, X)}
   =0. 
   \ee

    \begin{proof}[Proof of Lemma \ref{energy}]
    Since 
   $r$ is absolutely
    continuous and strictly positive on 
  	$[0,T]$, we know that
	$r_0 = \min\limits_{t\in[0,T]}r(t)>0$.
	For the nonnegative  function $s_+=\max
	\{s, 0\}$, the following properties
	are useful:
		for  every $c>0$ and $s_1,s_2\in\R$,
  	$$
  	|(s_1-c)_+|\le |s_1|,
  	\qquad
  	|(s_1-c)_+-(s_2-c)_+|\le |s_1-s_2|,
  	$$
	which imply that
  	 for almost every $t\in(0,T)$,
  	\be\label{enA p1a1}
  		\|(w(t)-r(t))_+\|
		\le\|w(t)\|
		\quad \text{and} \quad
  		\|(w(t)-r(t))_+\|_{L^p(\R^n)}
		\le\|w(t)\|_{L^p(\R^n)}.
		\ee
		In addition,  we have
  		\be\label{enA p1a2}
		  		\|(w(t)-r(t))_+\|_V \le\|w(t)\|_V,
		\ee
		which follows from the fact that
  	$$
  		\int_{\R^n}\int_{\R^n}
  		\frac{|(w(t,x)-r(t))_+-(w(t,y)-r(t))_+|^2}
  		{|x-y|^{n+2\alpha}}\,dx\,dy
  	\le
  		\int_{\R^n}\int_{\R^n}
  		\frac{|w(t,x)-w(t,y)|^2}{|x-y|^{n+2\alpha}}\,dx\,dy.
  	$$
	By 
     \eqref{enA p1a1}-\eqref{enA p1a2}   we see that 
     $ (w-r)_+\in L^2(0,T;V)\cap L^p(0,T;L^p(\R^n)).$
       
  	On the set $\{x:w(t,x)>r(t)\}$, we have
  	$0\le w(t,x)-r(t)
	\le w(t) \le 
	 r(t)^{-1}|w(t,x)|^2$,  and hence for 
  	$t\in[0,T]$,
  	\begin{equation}\label{enA p1}
  		\int_{\R^n}(w(t,x)-r(t))_+\,dx
  		\le\frac{\|w(t)\|^2}{r(t)}
  		\le\frac{\|w\|_{C([0,T],H)}^2}{r_0}.
  	\end{equation}
	By \eqref{enA p1} we see that
	the last term in \eqref{energy1}
	  belongs to $L^1(0,T)$.
	 
	 Next we show
	  $(w-r)_+\in C([0,T],H)$.  Note that
	   if $c,d>0$ and
  	$\xi\in H$, then
  	$$
  	|(\xi-c)_+-(\xi-d)_+|
  	\le |c-d|\,\1_{\{\xi>\min\{c,d\}\}},
  	\qquad
  	|\{\xi>\min\{c,d\}\}|
  	\le\frac{\|\xi\|^2}{\min\{c,d\}^2},
  	$$
	which implies that
	 for $s,t\in[0,T]$,
  	\begin{align}
  		\|(w(t)-r(t))_+-(w(s)-r(s))_+\|
  \le\|w(t)-w(s)\|
  		+\frac{|r(t)-r(s)|}{r_0}\,\|w(s)\|.
  		\label{enA p2}
  	\end{align}
	Since $w\in C([0,T], H)$ and $r\in C([0,T], \R)$,
	it follows from \eqref{enA p2}
	that 
	 $(w-r)_+\in C([0,T],H)$, and hence 
	   	\begin{equation}
		\label{enA p3}
  		E\in C([0,T]),\qquad
  		0\le E(t)\le\frac12\|w\|_{C([0,T],H)}^2.
  	\end{equation}
  	
	We now prove the right-hand side of \eqref{energy1}
	is integrable on $(0,T)$. Indeed, by 
   H\"older's inequality
   and
    \eqref{enA p1a1}-\eqref{enA p1}  we have
  	\begin{align}
  		&\int_0^T
  		|(f_1(t),(w(t)-r(t))_+)_{(V^*,V)}|\,dt
		+\int_0^T\int_{\R^n}
  		|f_2(t,x)|(w(t,x)-r(t))_+\,dx\,dt\notag\\
  		&\quad+\int_0^T |r'(t)|
  		\int_{\R^n}(w(t,x)-r(t))_+\,dx\,dt\notag\\
  		&\le
  		\|f_1\|_{L^2(0,T;V^*)}\|w\|_{L^2(0,T;V)}
  		+\|f_2\|_{L^q(0,T;L^q(\R^n))}
  		\|w\|_{L^p(0,T;L^p(\R^n))}\nonumber\\
		&+
  		\frac{\|w\|_{C([0,T],H)}^2}{r_0}
  		\|r'\|_{L^1(0,T)}<\infty.
		\label{enA p3a}
  	\end{align}
  	
	It remains to prove that 
	the right-hand side of  \eqref{energy1}
	 is the  weak
  	derivative of $E$ on $(0,T)$.
	Note that
	\begin{equation}\label{enA p4}
  		w(t)-w(s)=\int_s^t f_1(\tau)\,d\tau
  		+\int_s^t f_2(\tau)\,d\tau
  		\quad\hbox{in }(V\cap L^p(\R^n))^*,
  		\qquad 0\le s\le t\le T,
  	\end{equation}
	and  
  	\begin{equation}\label{enA p5}
  		r(t)-r(s)=\int_s^t r'(\tau)\,d\tau,
  		\qquad 0\le s\le t\le T.
  	\end{equation}
  	
  	 For convenience, we extend
  	 $f_1$, $f_2$, and $r^\prime$ by zero outside $(0,T)$,
	 and use the same notation for both the original functions
	 and their extensions.
	 Since
	 $f_1 \in L^2(0,T; V^*)$, by \eqref{cbi 2} we have
		 $$
  		\left\|\frac1\delta\int_0^\delta
  		\left ( f_1(\cdot \pm  s)-f_1\right )
		\,ds\right\|_{L^2(\R;V^*)}
  	\le\frac1\delta\int_0^\delta
  		\|f_1(\cdot  \pm  s)-f_1\|_{L^2(\R;V^*)}\,ds
  		\longrightarrow 0,
		\quad \text{as }  \  \delta \to 0,
  	$$ 	 
	and hence
		\be\label{enA p10}
  			\frac1\delta\int_0^\delta f_1(\cdot \pm  s)\,ds
  			\longrightarrow f_1 \ 
			 \text{ in }    L^2(\R;V^*),
			 \quad \text{ as } \delta \to 0.
		\ee	
		Similarly, since $f_2 \in 
		L^q(0,T; L^q(\R^n))$ and
		$r^\prime \in L^1(0,T; \R)$, by \eqref{cbi 2}
		we have 
		   \be\label{enA p11}
		   \frac1\delta\int_0^\delta f_2(\cdot \pm  s)\,ds
		  \longrightarrow f_2 \quad 
  		\text{in }L^q(\R;L^q(\R^n)),  \quad \text{ as } \delta \to 0,
		\ee
		and
		\be\label{enA p12}
  			\frac1\delta\int_0^\delta r^\prime (\cdot \pm  s)\,ds
  			\longrightarrow  r^\prime  \ 
			 \text{ in }    L^1(\R; \R),
			 \quad \text{ as } \delta \to 0.
		\ee			  	
		
	 Note that the
   function $s\mapsto\frac12(s_+)^2$ is convex and continuously
  	differentiable with derivative $s_+$, and hence
	$$
	\frac12(s_2)_+^2-\frac12(s_1)_+^2
	=\int_{s_1}^{s_2}
	s_+ ds,	$$
	which implies that
	\begin{equation}\label{enA p13}
  		(s_1)_+(s_2-s_1)
  		\le \frac12(s_2)_+^2-\frac12(s_1)_+^2
  		\le (s_2)_+(s_2-s_1),
  		\qquad s_1,s_2\in\R.
  	\end{equation}
  	Let
	 $\delta\in(0,T)$ be fixed.
	 First  applying  \eqref{enA p13} with
  	$s_1=w(t,x)-r(t)$ and
  	$s_2=w(t+\delta,x)-r(t+\delta)$,
	and then   integrating over $\R^n$ 
	 we obtain
	 	\begin{align}
  		\frac{E(t+\delta)-E(t)}{\delta}
  		&\ge
  		\left(\frac{w(t+\delta)-w(t)}{\delta},
  		(w(t)-r(t))_+\right)\notag\\
  		&\quad-\frac{r(t+\delta)-r(t)}{\delta}
  		\int_{\R^n}(w(t,x)-r(t))_+\,dx,
  		\label{enA p14}
  	\end{align}
  	and
  	\begin{align}
  		\frac{E(t+\delta)-E(t)}{\delta}
  		&\le
  		\left(\frac{w(t+\delta)-w(t)}{\delta},
  		(w(t+\delta)-r(t+\delta))_+\right)\notag\\
  		&\quad-\frac{r(t+\delta)-r(t)}{\delta}
  		\int_{\R^n}(w(t+\delta,x)-r(t+\delta))_+\,dx.
  		\label{enA p15}
  	\end{align}
	
	Next, we will take the limit of
	\eqref{enA p14} and \eqref{enA p15}
	as $\delta \to 0$
	by using \eqref{enA p4}-\eqref{enA p5}.
	To that end, we  choose   
  	a nonnegative continuous 
	function
	$\eta$
	  on $(0,T)$ with
	compact support (i.e., $\eta\in C _c (0,T)$)
	  and extend it by zero to
  	$\R$. 
	Let  $\delta>0$ be a positive number less than
  	half the distance of $\operatorname{supp}\eta$ to
  	$\{0,T\}$.  Multiplying \eqref{enA p14} by $\eta(t)$,
  	by  \eqref{enA p4}-\eqref{enA p5},
  	we obtain
  	\begin{align}
  		&\int_0^{T-\delta}\eta(t)
  		\frac{E(t+\delta)-E(t)}{\delta}\,dt\notag\\
  		&\quad\ge
  		\int_0^T\eta(t)
  		\left(\frac1\delta\int_t^{t+\delta}f_1(s)\,ds,
  		(w(t)-r(t))_+\right)_{(V^*,V)}\,dt\notag\\
  		&\qquad+
  		\frac1\delta\int_0^T\eta(t)\int_t^{t+\delta}
  		\int_{\R^n}f_2(s,x)(w(t,x)-r(t))_+\,dx\,ds\,dt\notag\\
  		&\qquad-
  		\int_0^T\eta(t)
  		\left(\frac1\delta\int_t^{t+\delta}r'(s)\,ds\right)
  		\int_{\R^n}(w(t,x)-r(t))_+\,dx\,dt.
  		\label{enA p16}
  	\end{align}
	Similarly, 
	 by multiplying \eqref{enA p15} by $\eta(t)$,
  	using  \eqref{enA p4}-\eqref{enA p5},
	after change of variables
  	we obtain  	 
  	\begin{align}
  		&\int_0^{T-\delta}\eta(t)
  		\frac{E(t+\delta)-E(t)}{\delta}\,dt\notag\\
  		&\quad\le
  		\int_0^T\eta(t-\delta)
  		\left(\frac1\delta\int_{t-\delta}^t f_1(s)\,ds,
  		(w(t)-r(t))_+\right)_{(V^*,V)}\,dt\notag\\
  		&\qquad+
  		\frac1\delta\int_0^T\eta(t-\delta)\int_{t-\delta}^t
  		\int_{\R^n}f_2(s,x)(w(t,x)-r(t))_+\,dx\,ds\,dt\notag\\
  		&\qquad-
  		\int_0^T\eta(t-\delta)
  		\left(\frac1\delta\int_{t-\delta}^t r'(s)\,ds\right)
  		\int_{\R^n}(w(t,x)-r(t))_+\,dx\,dt.
  		\label{enA p17}
  	\end{align}
	Note that the
  	 support condition on $\eta$
	 is used in 
	 \eqref{enA p17}
	 to ensure  the corresponding   integration
  	limits.
	
	We now prove, as $\delta
	\to 0$, the right-hand sides
	of \eqref{enA p16} and
	\eqref{enA p17}
	have the same limit.
	We first deal with \eqref{enA p17}.
	By 
	\eqref{enA p1a2}
	and
		\eqref{enA p10} we have
	$$
	\left |
			\int_0^T
  			\left(\eta(t-\delta)\frac1\delta\int_{t-\delta}^t f_1(s)\,ds,
  			   			(w(t)-r(t))_+\right)_{(V^*,V)}dt
				-
				\int_0^T 
				\left ( 
				\eta(t)f_1(t),
				(w(t)-r(t))_+\right)_{(V^*,V)}dt	
				\right |	
				$$
	$$
	=
	\left |
			\int_0^T
  			\left(\eta(t-\delta)\frac1\delta\int_{t-\delta}^t f_1(s)\,ds
  			   	-\eta(t)f_1(t),
						(w(t)-r(t))_+\right)_{(V^*,V)}dt
				\right |	
				$$	
				$$
	\le
	\left |
			\int_0^T
  			\left(\eta(t-\delta)
			\left (
			\frac1\delta\int_{t-\delta}^t f_1(s)\,ds
  			   	-f_1(t)
				\right ),
						(w(t)-r(t))_+\right)_{(V^*,V)}dt
				\right |	
				$$	
			$$
	+
	\left |
			\int_0^T
			\left (
  			\left(\eta(t-\delta)
			-\eta (t) \right )
		f_1(t),
						(w(t)-r(t))_+\right)_{(V^*,V)}dt
				\right |	
				$$	
$$
\le 
\|\eta\|_{L^\infty(\R, \R)}
\|w\|_{L^2(0,T;V)}
			\left \|\frac1\delta\int_0^\delta f_1(\cdot-s)\,ds-f_1
  			\right\|_{L^2(\R;V^*)}
			$$						
		\be\label{enA p18}
+
\|\eta (\cdot -\delta)
 -\eta \|_{L^\infty(\R, \R)}
\|w\|_{L^2(0,T;V)}
			\left \|f_1
  			\right\|_{L^2(\R;V^*)}
			\to 0,
			\ee
			as $\delta \to 0$.

	Similarly,
	by 
	\eqref{enA p1a1}
	and
		\eqref{enA p11} we have
$$
	\left |
			\int_0^T
  			\left(\eta(t-\delta)\frac1\delta\int_{t-\delta}^t f_2(s)\,ds,
  			   			(w(t)-r(t))_+\right)_{(L^q,L^p)}dt
				-
				\int_0^T 
				\left ( 
				\eta(t)f_2(t),
				(w(t)-r(t))_+\right)_{(L^q,L^p)}dt	
				\right |	
				$$
				
	$$
	=
	\left |
			\int_0^T
  			\left(\eta(t-\delta)\frac1\delta\int_{t-\delta}^t f_2(s)\,ds
  			   	-\eta(t)f_2(t),
						(w(t)-r(t))_+\right)_{(L^q(\R^n),L^p(\R^n))}dt
				\right |	
				$$	
				$$
	\le
	\left |
			\int_0^T
  			\left(\eta(t-\delta)
			\left (
			\frac1\delta\int_{t-\delta}^t f_2(s)\,ds
  			   	-f_2(t)
				\right ),
						(w(t)-r(t))_+\right)
						_{(L^q(\R^n),L^p(\R^n))}dt
										\right |	
				$$	
			$$
	+
	\left |
			\int_0^T
			\left (
  			\left(\eta(t-\delta)
			-\eta (t) \right )
		f_2(t),
						(w(t)-r(t))_+\right)
						_{(L^q(\R^n),L^p(\R^n))}dt				\right |	
				$$	
$$
\le 
\|\eta\|_{L^\infty(\R, \R)}
\|w\|_{L^p(0,T;L^p(\R^n) )}
			\left \|\frac1\delta\int_0^\delta f_2(\cdot-s)\,ds-f_2
  			\right\|_{L^q(\R;L^q(\R^n) )}
			$$						
\be\label{enA p19}
+
\|\eta (\cdot -\delta)
 -\eta \|_{L^\infty(\R, \R)}
\|w\|_{L^p(0,T;L^p(\R^n) )}
			\left \|f_2
  			\right\|_{L^q(\R;L^q(\R^n) )}
			\to 0,
			\ee
			as $\delta \to 0$.

	 For the term involving $r^\prime$, 
	 by
	 \eqref{enA p1} and  \eqref{enA p12} we have
  	$$
  		\left |
		\int_0^T
  		\left (\eta(t-\delta)\frac1\delta\int_{t-\delta}^t r^\prime (s)\,ds
  		-\eta(t)r^\prime (t)\right )
  		\int_{\R^n}(w(t,x)-r(t))_+\,dx\,dt\right |
		$$
		$$
  		\le
  		r^{-1}_0 {\|w\|_{C([0,T],H)}^2}
		\|\eta\|_{L^\infty(\R, 
		\R)}
  		\left\|
		\frac1\delta\int_0^\delta r ^\prime (\cdot-s)\,ds-r^\prime
  		\right\|_{L^1(\R, \R)}
  		 $$
		\be\label{enA p20}
  		+
  		r^{-1}_0 {\|w\|_{C([0,T],H)}^2}
  		\|\eta(\cdot-\delta)-\eta\|_{L^\infty(\R, \R)}
  		\|r^\prime \|_{L^1(\R, \R)}
  	\to 0.
	\ee
	It follows from
	\eqref{enA p17}-\eqref{enA p20}
	that
	 \begin{align}
  		&\limsup_{\delta \to 0}
  		\int_0^{T-\delta}\eta(t)
  		\frac{E(t+\delta)-E(t)}{\delta}\,dt
  	\le
  		\int_0^T\eta(t)(f_1(t),(w(t)-r(t))_+)_{(V^*,V)}\,dt\notag\\
  		&\qquad+\int_0^T\eta(t)\int_{\R^n}
  		f_2(t,x)(w(t,x)-r(t))_+\,dx\,dt
		-\int_0^T\eta(t)r^\prime (t)
  		\int_{\R^n}(w(t,x)-r(t))_+\,dx\,dt.
  		\label{enA p30}
  	\end{align}
	Similarly, by taking the limit of \eqref{enA p16} as $\delta
	\to 0$, we  obtain  the reverse inequality of
	\eqref{enA p30}, and thus 
	for every  nonnegative $\eta
	\in C_c(0,T)$ we have
	 \begin{align}
  		&\lim_{\delta \to 0}
  		\int_0^{T-\delta}\eta(t)
  		\frac{E(t+\delta)-E(t)}{\delta}\,dt
  	=
  		\int_0^T\eta(t)(f_1(t),(w(t)-r(t))_+)_{(V^*,V)}\,dt\notag\\
  		&\qquad+\int_0^T\eta(t)\int_{\R^n}
  		f_2(t,x)(w(t,x)-r(t))_+\,dx\,dt
		-\int_0^T\eta(t)r^\prime (t)
  		\int_{\R^n}(w(t,x)-r(t))_+\,dx\,dt.
  		\label{enA p31}
  	\end{align}	
	
	Suppose $\eta \in C^\infty_c(0,T)$ is an arbitrary 
	smooth test function on $(0,T)$.
	Then $\eta =   \eta_+ - (-\eta)_+ $
	where $\eta_+$ and $(-\eta)_+$ are both
	nonnegative test functions  in $C_c(0,T)$.
	   Applying \eqref{enA p31} to these
  	two functions and subtracting
	the resulting equalities, we infer that
	  \eqref{enA p31}  is valid for every
  	$\eta\in C_c^\infty(0,T)$.
	
	On the other hand,
	for every $\eta\in C_c^\infty(0,T)$,
   by   change of variables
   and the  dominated convergence theorem
      we get
  	\begin{equation}\label{enA p32}
  		\int_0^{T-\delta}\eta(t)
  		\frac{E(t+\delta)-E(t)}{\delta}\,dt
  		=\int_0^T E(t)\frac{\eta(t-\delta)-\eta(t)}{\delta}\,dt
  		\longrightarrow-\int_0^T E(t)\eta'(t)\,dt.
  	\end{equation}
  	  It follows from \eqref{enA p31}-\eqref{enA p32} that
	  for every $\eta\in C_c^\infty(0,T)$,
	     	\begin{align}
  		-\int_0^T E(t)\eta'(t)\,dt
  		&=\int_0^T\eta(t)(f_1(t),(w(t)-r(t))_+)_{(V^*,V)}\,dt\notag\\
  		&\quad+\int_0^T\eta(t)\int_{\R^n}
  		f_2(t,x)(w(t,x)-r(t))_+\,dx\,dt\notag\\
  		&\quad-\int_0^T\eta(t)r'(t)
  		\int_{\R^n}(w(t,x)-r(t))_+\,dx\,dt,
  		\label{enA p33}
  	\end{align}
  	By \eqref{enA p3a}, the expression on the right-hand side
	of \eqref{enA p33} is the
  	pairing of $\eta$ with an $L^1(0,T)$ function. 
	 Since $E\in L^1(0,T)$, by 
  	\eqref{enA p33} we find that  
  	$E\in W^{1,1}(0,T)$ and
	satisfies the equation
	\eqref{energy1},
	which completes the proof
	of Lemma \ref{energy}.
 \end{proof}

%
%
%
%
%

\end{document}